\documentclass[12pt]{article}

\usepackage{amsmath}
\usepackage{amssymb}
\usepackage{amsthm}
\usepackage{amsfonts}
\usepackage{latexsym}
\usepackage{cite}
\usepackage{psfrag}
\usepackage{epsfig}
\usepackage{graphicx}
\usepackage{mathrsfs}
\usepackage{url}
\usepackage{color}
\usepackage{cancel}
 \usepackage{eufrak}
 \usepackage[pdftex]{lscape}
 \usepackage{lscape}
 \usepackage{mathabx}

\usepackage{multirow}
\usepackage{multicol}
\usepackage{hyperref}

\usepackage{scalerel}
\usepackage{tikz}
\usetikzlibrary{svg.path}

\definecolor{orcidlogocol}{HTML}{A6CE39}
\tikzset{
  orcidlogo/.pic={
    \fill[orcidlogocol] svg{M256,128c0,70.7-57.3,128-128,128C57.3,256,0,198.7,0,128C0,57.3,57.3,0,128,0C198.7,0,256,57.3,256,128z};
    \fill[white] svg{M86.3,186.2H70.9V79.1h15.4v48.4V186.2z}
                 svg{M108.9,79.1h41.6c39.6,0,57,28.3,57,53.6c0,27.5-21.5,53.6-56.8,53.6h-41.8V79.1z M124.3,172.4h24.5c34.9,0,42.9-26.5,42.9-39.7c0-21.5-13.7-39.7-43.7-39.7h-23.7V172.4z}
                 svg{M88.7,56.8c0,5.5-4.5,10.1-10.1,10.1c-5.6,0-10.1-4.6-10.1-10.1c0-5.6,4.5-10.1,10.1-10.1C84.2,46.7,88.7,51.3,88.7,56.8z};
  }
}

\newcommand\orcidicon[1]{\href{https://orcid.org/#1}{\mbox{\scalerel*{
\begin{tikzpicture}[yscale=-1,transform shape]
\pic{orcidlogo};
\end{tikzpicture}
}{|}}}}

\usepackage{hyperref} 

\usepackage{longtable}

\makeatletter
\@addtoreset{equation}{section}
\makeatother

\makeatletter
\@addtoreset{table}{section}
\makeatother

\newtheorem{theorem}{Theorem}[section]
\newtheorem{lemma}[theorem]{Lemma}
\newtheorem{corollary}[theorem]{Corollary}
\theoremstyle{definition}
\newtheorem{definition}[theorem]{Definition}
\newtheorem{notation}[theorem]{Notation}
\newtheorem{remark}[theorem]{Remark}

\newcommand\Ic{\mathcal{I}}
\newcommand\Lc{\mathcal{L}}
\newcommand\Oc{\mathcal{O}}

\newcommand\Cs{\mathscr{C}}
\newcommand\Ls{\mathscr{L}}
\newcommand\Ms{\mathscr{M}}
\newcommand\Ns{\mathscr{N}}
\newcommand\Os{\mathscr{O}}

\newcommand\PG{\mathrm{PG}}

\newcommand\RC{\mathrm{RC}}
\newcommand\RA{\mathrm{RA}}
\newcommand\Tr{\mathrm{T}}
\newcommand\IC{\mathrm{IC}}
\newcommand\IA{\mathrm{IA}}
\newcommand\UG{\mathrm{U\Gamma}}
\newcommand\UnG{\mathrm{Un\Gamma}}
\newcommand\EG{\mathrm{E\Gamma}}
\newcommand\EnG{\mathrm{En\Gamma}}
\newcommand\Ar{\mathrm{A}}
\newcommand\EA{\mathrm{EA}}
\newcommand\TO{\mathrm{TO}}

\newcommand\Mb{\mathbf{M}}
\newcommand\Pb{\mathbf{P}}

\newcommand\F{\mathbb{F}}

\newcommand\Ak{\mathfrak{A}}
\newcommand\Lk{\mathfrak{L}}
\newcommand\Mk{\mathfrak{M}}
\newcommand\Pk{\mathfrak{P}}
\newcommand\pk{\mathfrak{p}}

\newcommand\T{\text}
\newcommand\db{\displaybreak[3]}

\newcommand\bez{$\widebar{0}$}
\newcommand\beo{$\widebar{1}$}
\newcommand\bet{$\widebar{2}$}
\newcommand\bethr{$\widebar{3}$}
\newcommand\bef{$\widebar{4}$}
\newcommand\befv{$\widebar{5}$}

\newcommand\wz{$\widehat{0}$}
\newcommand\wo{$\widehat{1}$}
\newcommand\wt{$\widehat{2}$}
\newcommand\wthr{$\widehat{3}$}
\newcommand\wf{$\widehat{4}$}

\newcommand\ws{$\widehat{6}$}
\newcommand\wsv{$\widehat{7}$}

\newcommand\tz{$\widetilde{0}$}
\newcommand\tone{$\widetilde{1}$}
\newcommand\ttwo{$\widetilde{2}$}
\newcommand\tthr{$\widetilde{3}$}

\newcommand\pl{\,+\,}
\newcommand\el{\,=\,}

\begin{document}
\title{Orbits and incidence matrices for points, planes and lines regarding the twisted cubic in $\PG(3,q)$, $q=2,3,4$
\date{}
}
\maketitle
\begin{center}
{\sc Alexander A. Davydov \orcidicon{0000-0002-5827-4560}}\\
 {\sc\small Kharkevich Institute for Information Transmission Problems}\\
 {\sc\small Russian Academy of Sciences,
Moscow, 127051, Russian Federation}\\
 \emph{E-mail address:} alexander.davydov121@gmail.com\medskip\\
 {\sc Stefano Marcugini \orcidicon{0000-0002-7961-0260} and
 Fernanda Pambianco \orcidicon{0000-0001-5476-5365}}\\
 {\sc\small Department of  Mathematics  and Computer Science,   University of Perugia,}\\
 {\sc\small Perugia, 06123, Italy}\\
 \emph{E-mail address:} \{stefano.marcugini, fernanda.pambianco\}@unipg.it
\end{center}

\textbf{Abstract.}
In the three-dimensional projective space $\mathrm{PG}(3,q)$
over the finite field $\mathbb{F}_q$ with $q$ elements, we consider the normal rational curve
known as a twisted cubic and the projectivity group $G_q$ that fixes it.
For $q=2,3,4$, we solve the open problems of classifying the orbits
of points, planes, and lines under $G_q$ and of determining
the corresponding incidence matrices between points, planes, and lines
partitioned into these orbits.



\section{Introduction}
Let $\F_q$ be the Galois field with $q$ elements, $\F_q^*=\F_q\setminus\{0\}$, and $\F_q^+=\F_q\cup\{\infty\}$. Let $\PG(N,q)$ be the $N$-dimensional projective space over $\F_q$. An $n$-arc in  $\PG(N,q)$ is a set of $n$ points such that no $N +1$ points belong to the same hyperplane of $\PG(N,q)$. In $\PG(N,q)$, a normal rational curve is a $(q+1)$-arc projectively equivalent to the arc $\{(t^N,t^{N-1},\ldots,t^2,t,1):t\in \F_q^+\}$. In $\PG(3,q)$, the normal rational curve is called a \emph{twisted cubic} \cite{Hirs_PG3q,HirsThas-2015}. For an introduction to projective geometry over finite fields, see \cite{Hirs_PG3q,Hirs_PGFF,HirsStor-2001,HirsThas-2015}.

The twisted cubic has many interesting properties and is connected with various combinatorial and applied problems; see, e.g., \cite{BDMP-TwCub,BallicoCos,BlokPelSzo,BonPolvTwCub,BrHirsTwCub,CLPolvT_Spr,CerPavDM,CosHirsStTwCub,%
DMP_RSCoset,DMP_PlLineIncJG,DMP_PointLineInc,DMP_OrbLineO6MJOM,DMP_IncO6JG,DMP_MJOM2025,%
DMP_OrbLineMJOM,GiulVincTwCub,GulLav,Hirs_PG3q,HirsStor-2001,HirsThas-2015,KorchLanzSon,LunarPolv,ZanZuan2010,%
KaPatPradOrbits,KaPrad-q3orbits,KaPradInciden} and the references therein. In particular, properties of the twisted cubic have been used to study spreads and partial spreads in $\PG(3,q)$ \cite{BallicoCos,BrHirsTwCub,BonPolvTwCub,CLPolvT_Spr,LunarPolv}, to construct optimal multiple covering codes \cite{BDMP-TwCub}, to obtain the weight distributions of cosets and their leaders for Reed--Solomon codes \cite{BlokPelSzo,DMP_RSCoset}, and to consider 2-level and 3-level secret sharing schemes \cite{GiulVincTwCub,KorchLanzSon}.

In investigations of the twisted cubic, the following problems are important:

\textbf{Problem A.} Describe exactly the orbits of points, planes, and lines under the action of the projectivity group $G_q$ fixing the twisted cubic.

\textbf{Problem B.} Obtain the matrices of the incidences between points, planes, and lines partitioned into the orbits under $G_q$.
In other words, Problem B consists of obtaining the point-plane, point-line, and plane-line incidence matrices, whose submatrices are associated with the corresponding orbits under $G_q$.

Problems A and B have been considered in many works; see, e.g., \cite{BDMP-TwCub,BlokPelSzo,BrHirsTwCub,CerPavDM,%
DMP_PlLineIncJG,DMP_PointLineInc,DMP_OrbLineO6MJOM,DMP_IncO6JG,DMP_MJOM2025,DMP_OrbLineMJOM,GulLav,Hirs_PG3q,%
KaPatPradOrbits,KaPrad-q3orbits,KaPradInciden} and the references therein. The results are different for the fields $\F_q$ with $q\ge5$
and for $q=2,3,4$. In part, this can be explained by the fact that for all $q\ge5$, we have $G_q\cong PGL(2,q)$, where $PGL(2,q)$ is the group of projectivities of $\PG(1,q)$, whereas for $q=2,3,4$  the structures of $G_q$ are different.

For the fields $\F_q$, for $q\ge5$, the following results have been obtained in the literature.

In \cite{Hirs_PG3q}, the orbits of planes and points under $G_q$ are
described in detail; see Section~\ref{sec_prelimin}, Theorem \ref{th2:Hirs}(ii)--(iv). For all $q\ge5$, the point-plane incidence matrices of $\PG(3,q)$, based on these results, are given in \cite{BDMP-TwCub}, where the numbers of distinct planes through distinct
points and, conversely, the numbers of distinct points lying in distinct planes are obtained. (By ``distinct planes'' we mean ``planes from distinct orbits'', and similarly for points and lines.)

In \cite{BrHirsTwCub,CerPavDM,Hirs_PG3q}, the lines in $\PG(3,q)$ are partitioned into classes, denoted by $\Oc_i$ and $\Oc'_i$, each of which is a union of line orbits under $G_q$; see also Section \ref{sec_prelimin}, Theorem \ref{th2:LineClasses}. Apart from one class (denoted by $\Oc_6$), the number, sizes, and structure of the orbits forming these unions are
independently considered by different methods in \cite[Sections 3, 8]{DMP_OrbLineMJOM} (for all $q\ge5$), \cite[Section 7]{BlokPelSzo}
(for all $q\ge23$), and \cite{GulLav} (for finite fields of characteristic $>3$); see also the references therein. The results on line orbits from \cite{BlokPelSzo,DMP_OrbLineMJOM,GulLav} are in accordance with each other, but the representations of the orbits in these papers are different. Also, in \cite{DMP_OrbLineMJOM}, the stabilizer groups of the considered orbits are obtained and described in detail, whereas in \cite{BlokPelSzo,GulLav} the stabilizer groups are not considered.

Using the representations of the line orbits in \cite{DMP_OrbLineMJOM} and the plane and point orbits in \cite{Hirs_PG3q},  the point-line and plane-line incidence matrices of $\PG(3,q)$ are given in \cite{DMP_PointLineInc} and \cite{DMP_PlLineIncJG}, respectively.  In \cite{DMP_PointLineInc}, apart from the lines in the class $\Oc_6$, for all $q\ge5$, the numbers of distinct points lying on distinct lines and, conversely, the numbers of distinct lines through distinct points are obtained. Similarly, in \cite{DMP_PlLineIncJG}, apart from $\Oc_6$, for all $q\ge5$, the numbers of distinct planes through distinct lines and, conversely,
the numbers of distinct lines lying in distinct planes are given.

In \cite{GulLav}, apart from the lines in the class $\Oc_6$, for odd $q\not\equiv0\pmod 3$, the numbers of distinct planes through distinct lines (called ``the plane orbit distribution of a line'') and the numbers of distinct points lying on distinct
lines (called ``the point orbit distribution of a line'') in $\PG(3,q)$ are obtained. For finite fields of characteristic $>3$, the results of \cite{GulLav} on ``the point orbit distribution of a line'' and ``the plane orbit distribution of a line'' are in accordance with those of \cite{DMP_PointLineInc,DMP_PlLineIncJG} on the point-line and plane-line incidence matrices.

It should be noted that in \cite[Section 8, Theorem 8.1, Conjecture 8.2]{DMP_OrbLineMJOM}, for the all fields $\F_q$, $q\ge5$, a detailed conjecture on the number of line orbits, belonging to the class $\Oc_6$, and on the sizes of these orbits is formulated; for $5\le q\le37$ and $q=64$, the conjecture was proved by an exhaustive
computer search. Furthermore, in the papers \cite{CerPavDM,DMP_OrbLineO6MJOM,DMP_MJOM2025,KaPatPradOrbits,KaPrad-q3orbits}, this conjecture is proved for all $q\ge5$.

In \cite{DMP_OrbLineO6MJOM}, for all even and odd $q\ge5$, a so-called family $\Os_\mu$ of line orbits of the class $\Oc_6$ is obtained using a line family, called \emph{$\ell_\mu$-lines}, where a parameter $\mu$ runs over $\F_q^*\setminus\{1\}$. Also, one more  orbit $\Os_\Lc$, based on a line $\Lc$ with another description, is given. The orbits $\Os_\mu$ and $\Os_\Lc$ are based on the analysis of the stabilizer groups of the corresponding lines. These orbits include an essential part of all $\Oc_6$ orbits; e.g., they include about one-half and  one-third of all lines of $\Oc_6$ for even $q$ and for $q\equiv 0 \pmod3$, respectively.

 In \cite{DMP_IncO6JG}, using the properties of the orbits $\Os_\mu$ and $\Os_\Lc$ from \cite{DMP_OrbLineO6MJOM}, all the plane-line and point-line incidence matrices connected with these orbits are obtained.

In \cite{DMP_MJOM2025}, the approaches of \cite{DMP_OrbLineO6MJOM,DMP_IncO6JG} are continued and developed. A new family of lines $\Ls_\rho$, where $\rho$ is a parameter running over $\F_q^*$, is proposed. This family is a generalization of the line $\Lc$ from \cite{DMP_OrbLineO6MJOM}. For $q\not\equiv 0 \pmod3$ the lines $\Ls_\rho$ belong to the class $\Oc_6$. The investigation of the stabilizer groups of the lines $\Ls_\rho$ for all $q$ and $\rho$ allows to calculate the sizes of the orbits under $G_q$ containing the lines $\Ls_\rho$. The point-line and plane-line incidence submatrices related to these orbits are also obtained in \cite{DMP_MJOM2025}.

In  $\PG(3,q)$, for $q=2^n$, $n\ge3$,
the $(q+1)$-arc $\mathcal{A}_{h,q}=\{(1,t,t^{2^h},t^{2^h+1})\,|\,t\in\F_q^+\}$, with $gcd(n,h)=1$, has been considered in \cite{CerPavDM}. In this paper, we are interested in the case $h=1$, since $\mathcal{A}_{1,q}$, $q=2^n\ge8$, is the twisted cubic with projective stabilizer $G_q\cong PGL(2,q)$. In \cite{CerPavDM}, for even $q\ge8$, all point, plane, and line orbits under $G_q$, including the orbits of  the class $\Oc_6$, are classified and the conjecture of \cite[Theorem 8.1, Conjecture 8.2]{DMP_OrbLineMJOM} is proved. Also, the point-plane and point-line incidence matrices with respect to the $G_q$-orbits are obtained; however, the plane-line incidence matrix is not considered.

In \cite{KaPatPradOrbits}, for a field $\F_q$ of characteristic $>3$, the conjecture of \cite[Theorem 8.1, Conjecture 8.2]{DMP_OrbLineMJOM} on the sizes and the number of orbits in the class  $\Oc_6$ has been proved. To this end the open problem of classifying binary quartic forms over $\F_q$ into $G_q$-orbits is solved and used. Also, the Pl\"{u}cker embedding for the Klein quadric is applied. The point-line incidence matrix, related to the considered orbits, is obtained in \cite{KaPradInciden}, where the following is also noted:  if the characteristic of the field $\F_q$ is not equal to  $3$,  then there is a symplectic polarity associated with the twisted cubic that maps each point of the twisted cubic to its osculating plane. Therefore, obtaining results for the point-line incidence matrices is equivalent to obtaining the corresponding results for the plane-line incidence matrices, since the latter can be derived via the polarity.

In \cite{KaPrad-q3orbits}, the problem of classifying the lines of $\PG(3,q)$ into orbits of the group
$PGL(2,q)$ is solved for a field $\F_q$ of characteristic $3$. Similarly to \cite{KaPatPradOrbits}, the problem is reduced to the classification of binary quartic forms into $G_q$-orbits. Also, the corresponding point-line and line-plane incidence matrices are obtained.

In \cite{KaPatPradOrbits,KaPrad-q3orbits,KaPradInciden}, the methods and approaches are different from those in the papers \cite{CerPavDM,DMP_PointLineInc,DMP_PlLineIncJG,DMP_OrbLineMJOM,DMP_OrbLineO6MJOM,DMP_IncO6JG,DMP_MJOM2025}.

Thus, for all fields $\F_q$ with $q\ge5$, including $q=2^m\ge8$  and $q=3^m\ge9$,
Problem A is solved completely, whereas Problem~B remains partially open. More exactly, for all the fields $\F_q$, $q\ge5$, the point-plane and point-line incidence matrices are completely obtained, the same holds for plane-line incidence matrices when $q\equiv0\pmod3$; however, for $q\not\equiv0\pmod3$, the plane-line incidence matrices, that include lines of the class $\Oc_6$, are not fully obtained, see \cite{DMP_OrbLineO6MJOM,DMP_IncO6JG,DMP_MJOM2025}.

 For the fields $\F_q$ with $q=2,3,4$, in the literature see, e.g., \cite{DMP_PointLineInc,DMP_PlLineIncJG,DMP_OrbLineMJOM}, the orbits and incidence matrices are considered only for special subgroups of $G_q$, called ``critical'' in this paper (Definition \ref{def3:critic}, Lemma \ref{lem3:coincid}), but for the group $G_q$ itself, the point, plane, and line orbits under it and the corresponding incidence matrices have not been considered; thus, \emph{for the fields $\F_q$ with  $q=2,3,4$, Problems A and B remain open.}

 \emph{In this paper, we solve Problems A and B for the fields $\F_q$ with $q=2,3,4$.} We classify the points, planes, and lines orbits under $G_q$ and obtain the corresponding matrices of the incidences between points, planes, and lines partitioned into the orbits.

The paper is organized as follows. Section \ref{sec_prelimin} contains preliminaries and notation. Useful relations are given in Section \ref{sed3:useful}. Orbits and incidence matrices for $q = 2,3$, and $4$ are considered in Sections  \ref{sec4:q=2},
\ref{sec5:q=3}, and \ref{sec6:q=4}, respectively.

\section{Preliminaries and notation}\label{sec_prelimin}
In this section, based on \cite[Chapters 15.2, 21, Appendix III]{Hirs_PG3q}, \cite[Section 6.4]{Hirs_PGFF}, \cite{Adriaensen,BDMP-TwCub,Block1967,BlokPelSzo,CerPavDM,DMP_PointLineInc,DMP_PlLineIncJG,DMP_OrbLineMJOM,%
DMP_OrbLineO6MJOM,DMP_IncO6JG,DMP_MJOM2025,GulLav,KaPatPradOrbits,KaPrad-q3orbits,KaPradInciden}, we define objects connected with the twisted cubic, and give their known properties that hold for $q\ge5$ (if we note this) or for all $q\ge2$ (if remarks are absent or we note this).

The space $\PG(3,q)$ contains $\theta_{3,q}$ points and planes, and $\beta_{3,q}$ lines,
\begin{equation}\label{eq2:theta-beta}
  \theta_{3,q}=q^3+q^2+q+1,~\beta_{3,q}=(q^2+1)(q^2+q+1).
\end{equation}
Let $\Pb(x_0,x_1,x_2,x_3)$ be a point of $\PG(3,q)$ with homogeneous coordinates $x_i\in\F_{q}$; the rightmost nonzero coordinate is equal to $1$. Let  $P(t)$, $t\in\F_q^+$, be a point such that
$P(t)=\Pb(t^3,t^2,t,1)\T{ if }t\in\F_q,~P(\infty)=\Pb(1,0,0,0)$.
Let $\Cs\subset\PG(3,q)$ be the \emph{twisted cubic} consisting of $q+1$ points.
We consider $\Cs$ in the \emph{canonical form}
\begin{equation}\label{eq2:cubic}
\Cs=\{P(t)\,|\,t\in\F_q^+\}.
\end{equation}
Let $\boldsymbol{\pi}(c_0,c_1,c_2,c_3)\subset\PG(3,q)$, $c_i\in\F_q$, be the plane with equation
$c_0x_0+c_1x_1+c_2x_2+c_3x_3=0.$ Let $\pi_\T{osc}(t)$ be
an \emph{osculating plane} through the point $P(t)$ such that
$\pi_\T{osc}(t)=\boldsymbol{\pi}(1,-3t,3t^2,-t^3)$  if $t\in\F_q$; $\pi_\T{osc}(\infty)=\boldsymbol{\pi}(0,0,0,1)$.
The $q+1$ osculating planes form the \emph{osculating developable} $\Gamma$ to $\Cs$ which is a \emph{pencil of planes} for $q\equiv0\pmod3$.

A \emph{chord} of $\Cs$ is a line through a pair of real points of $\Cs$ or a pair of complex conjugate points. In the last  case it is an \emph{imaginary chord}. If the real points are distinct, it is a \emph{real chord}. If the real points coincide with each other, it is a \emph{tangent} to $\Cs$. The tangent to $\Cs$ at the point $P(t)$ can be defined as a line with the coordinate vector $(t^4,2t^3,3t^2,t^2,-2t,1)$.

An \emph{axis} of $\Gamma$ is a line of $\PG(3,q)$ which is the meet of a pair of real planes or complex conjugate planes of $\Gamma$. In the last case it is an \emph{imaginary axis}. If the real planes are distinct it is a \emph{real axis}; if they coincide with each other it is a \emph{tangent} to $\Cs$.

For $q\not\equiv0\pmod3$, the null polarity $\Ak$ \cite[Theorem~21.1.2]{Hirs_PG3q}, \cite[Sections 2.1.5, 5.3]{Hirs_PGFF}, is given by
\begin{align}\label{eq2_null_pol}
&\Pb(x_0,x_1,x_2,x_3)\Ak=\boldsymbol{\pi}(x_3,-3x_2,3x_1,-x_0).
\end{align}
It interchanges $\Cs$ and $\Gamma$, and their corresponding chords and axes.

\begin{notation}\label{notation2:1}
Throughout the paper, we consider $q\equiv\xi\pmod3$, $\xi\in\{-1,0,1\}$.

The following notation is used for $q\ge2$.
\begin{align*}
  &G_q &&\T{the group of projectivities in } \PG(3,q) \T{ fixing }\Cs;\db  \\
  &\mathbf{Z}_n&&\T{cyclic group of order }n;\db  \\
  &\mathbf{S}_n&&\T{symmetric group of degree }n;\db  \\
&PGL(n,q)&&\T{the group of projectivities of }\PG(n-1,q);\db\\
&P\Gamma L(n,q)&&\T{the group of collineations of }\PG(n-1,q);\db\\
&PSL(n,q)&&\T{the group of projectivities of }\PG(n-1,q)\T{ of determinant one};\db\\
&A^{tr}&&\T{the transposed matrix of }A;\db \\
&\#S&&\T{the cardinality of a set }S;\db\\
&\overline{AB}&&\T{the line through the points $A$ and }B;\db\\
&\triangleq&&\T{the sign ``equality by definition''}.
\end{align*}
\hspace{4cm}\textbf{Types $\pi$ of planes:}
\begin{align*}
&\Gamma\T{-plane}  &&\T{an osculating plane of }\Gamma \T{ (it contains exactly 1 point of }\Cs);\db \\
&\overline{1_\Cs}\T{-plane}&&\T{a plane not in $\Gamma$ containing \emph{exactly} 1 point of }\Cs;\db \\
&d_\Cs\T{-plane}&&\T{a plane containing \emph{exactly} $d$ distinct points of }\Cs,~d=0,2,3;\db\\
&\Pk&&\T{the list of possible types $\pi$ of planes},~
\Pk\triangleq\{\Gamma,\overline{1_\Cs},0_\Cs,2_\Cs,3_\Cs\};\db\\
&\pi\T{-plane}&&\T{a plane of the type }\pi\in\Pk.
\end{align*}
\hspace{4cm}\textbf{Types $\pk$ of points:}
\begin{align*}
&\Cs\T{-point}&&\T{a point  of }\Cs;\db\\
&\mu_\Gamma\T{-point}&&\T{a point off $\Cs$ lying on \emph{exactly} $\mu$ distinct osculating planes},\db\\
&&&\mu_\Gamma\in\{0_\Gamma,1_\Gamma,3_\Gamma\}\T{ for }\xi\ne0, \mu_\Gamma\in\{(q+1)_\Gamma\}\T{ for }\xi=0;\db\\
&\Tr\T{-point}&&\T{a point off $\Cs$  on a tangent to $\Cs$ for any }\xi;\db\\
&\T{TO-point}&&\T{a point off $\Cs$ on a tangent and one osculating plane for }\xi=0;\db\\
&\RC\T{-point}&&\T{a point off $\Cs$  on a real chord;}\db\\
&\IC\T{-point}&&\T{a point on an imaginary chord (it always is off $\Cs$);}\\
&\Mk^{(\xi)}&&\T{the list of possible types $\pk$ of points},\db\\
&&&\Mk^{(\ne0)}\triangleq\{\Cs,0_\Gamma,1_\Gamma,3_\Gamma,\Tr,\RC,\IC\}\T{ for }\xi\ne0,\db\\
&&&\Mk^{(0)}\triangleq\{\Cs,(q+1)_\Gamma,\TO,\RC,\IC\}\T{ for }\xi=0;\db\\
&\pk\T{-point}&&\T{a point of the type }\pk\in\Mk^{(\xi)}.
\end{align*}
\hspace{4cm}\textbf{Types $\lambda$ of lines:}
\begin{align*}
&\RC\T{-line}&&\T{a real chord  of $\Cs$;}\db \\
&\RA\T{-line}&&\T{a real axis of $\Gamma$ for }\xi\ne0;\db \\
&\Tr\T{-line}&&\T{a tangent to $\Cs$};\db \\
&\IC\T{-line}&&\T{an imaginary chord  of $\Cs$;}\db \\
&\IA\T{-line}&&\T{an imaginary axis of $\Gamma$ for }\xi\ne0;\db \\
&\UG\T{-line}&&\T{a non-tangent unisecant in a $\Gamma$-plane;}\db \\
&\T{Un$\Gamma$-line}&&\T{a unisecant not in a $\Gamma$-plane (it is always non-tangent);}\db \\
&\T{E$\Gamma$-line}&&\T{an external line in a $\Gamma$-plane (it cannot be a chord);}\db \\
&\EnG\T{-line}&&\T{a line, external to the cubic $\Cs$, not in a $\Gamma$-plane, that is neither}\db\\
&&&\T{a chord nor an axis, see \cite[Lemma 21.1.4]{Hirs_PG3q} and its context;}\db \\
&\Ar\T{-line}&&\T{the axis of the pencil of all the $\Gamma$-planes for }\xi=0;\db \\
&\EA\T{-line}&&\T{an external line meeting the axis of $\Gamma$ for }\xi=0;\db \\
&\Lk^{(\xi)}&&\T{the list of possible types $\lambda$ of lines},\db\\
&&&\Lk^{(\ne0)}\triangleq\{\RC,\RA,\Tr,\IC,\IA,\UG,\UnG,\EG,\EnG\}\T{ for }\xi\ne0,\db\\
&&&\Lk^{(0)}\triangleq\{\RC,\Tr,\IC,\UG,\UnG, \EnG,\Ar,\EA\}\T{ for }\xi=0;\db\\
&\lambda\T{-line}&&\T{a line of the type }\lambda\in\Lk^{(\xi)}. \db\\
&&&\textbf{Types of orbits:}\db\\
&\Ns_i&&\T{the $i$-th orbit of planes under }G_q;\db\\
&\Ms_j&&\T{the $j$-th orbit of points under }G_q;\db\\
&\Lc_j&&\T{the $j$-th orbit of lines under }G_q.\db\\
&&&\textbf{Types of incidences matrices $\Ic^{\bullet}$ and submatrices $\Ic^{\bullet}_{ij}$:}\db\\
&\Ic^{\Pi\Pb}&&\T{plane-point incidence $\theta_{3,q}\times\theta_{3,q}$ matrix in }\PG(3,q);\db\\
&\Ic^{\Pi\Pb}_{ij}&&\T{an $\#\Ns_i\times\#\Ms_j$ submatrix of the matrix $\Ic^{\Pi\Pb}$ in }\PG(3,q),\\
&&&\T{given by orbits }\Ns_i,~\Ms_j;\db\\
&\Ic^{\Lambda\Pb}&&\T{line-point incidence $\beta_{3,q}\times\theta_{3,q}$ matrix in }\PG(3,q);\db\\
&\Ic^{\Lambda\Pb}_{ij}&&\T{an $\#\Lc_i\times\#\Ms_j$ submatrix of the matrix $\Ic^{\Lambda\Pb}$ in }\PG(3,q),\\
&&&\T{given by orbits }\Lc_i,~\Ms_j;\db\\
&\Ic^{\Lambda\Pi}&&\T{line-plane incidence $\beta_{3,q}\times\theta_{3,q}$ matrix in }\PG(3,q);\db\\
&\Ic^{\Lambda\Pi}_{ij}&&\T{an $\#\Lc_i\times\#\Ns_j$ submatrix of the matrix $\Ic^{\Lambda\Pi}$ in }\PG(3,q),\\
&&&\T{given by orbits }\Lc_i,~\Ns_j.
\end{align*}
\end{notation}

\begin{theorem}\label{th2:Hirs}
\emph{\cite[Chapter 21]{Hirs_PG3q}} Let $q\ge5$. The following properties of the twisted cubic $\Cs$ of \eqref{eq2:cubic} hold:

\textbf{\emph{(i)}}
We have $G_q\cong PGL(2,q)$, the group $G_q$ acts triply transitively on $\Cs$, and the matrix $\Mb$, corresponding to a projectivity of $G_q$, has the general form
  \begin{equation}\label{eq2:q ge5 matrixGq}
 \Mb=\left[
 \begin{array}{cccc}
 a^3&a^2c&ac^2&c^3\\
 3a^2b&a^2d+2abc&bc^2+2acd&3c^2d\\
 3ab^2&b^2c+2abd&ad^2+2bcd&3cd^2\\
 b^3&b^2d&bd^2&d^3
 \end{array}
  \right],~a,b,c,d\in\F_q,~ ad-bc\ne0.
  \end{equation}

\textbf{\emph{(ii)}}  Under $G_q$, there are 5 orbits $\Ns_j$ of planes:
\begin{align*}
&\Ns_1=\{\Gamma\T{-planes}\},~ \#\Ns_1=q+1;~\Ns_{2}=\{2_\Cs\T{-planes}\},~\#\Ns_2=q^2+q;\db\\
&\Ns_{3}=\{3_\Cs\T{-planes}\},~\#\Ns_3=(q^3-q)/6;~ \Ns_{4}=\{\overline{1_\Cs}\T{-planes}\},~ \#\Ns_4=(q^3-q)/2;\db\\
&\Ns_{5}=\{0_\Cs\T{-planes}\},~ \#\Ns_5=(q^3-q)/3.
\end{align*}

\textbf{\emph{(iii)}} For $q\not\equiv0\pmod 3$, under $G_q$, there are 5 orbits $\Ms_j$ of points:
\begin{align*}
&\Ms_1=\{\Cs\T{-points}\},~ \Ms_2=\{\Tr\T{-points}\},~ \Ms_3=\{3_\Gamma\T{-points}\},~ \Ms_4=\{1_\Gamma\T{-points}\},~\db\\
&\Ms_5=\{0_\Gamma\T{-points}\};~ \#\Ms_j=\#\Ns_j,~ j=1,~\ldots,~5.\db\\
&\T{If } q\equiv1\pmod 3 \T{ then } \Ms_3\cup\Ms_5=\{\RC\T{-points}\},~   \Ms_4=\{\IC\T{-points}\}.\db\\
&\T{If } q\equiv-1\pmod 3\T{ then }\Ms_3\cup\Ms_5=\{\IC\T{-points}\},~ \Ms_4=\{\RC\T{-points}\}.
\end{align*}

\textbf{\emph{(iv)}} For $q\equiv0\pmod 3$, under $G_q$, there are 5 orbits $\Ms_j$ of points:
\begin{align*}
& \Ms_1=\{\Cs\T{-points}\},\, \Ms_2=\{(q+1)_\Gamma\T{-points}\},~ \#\Ms_1=\#\Ms_2=q+1;~
\Ms_3=\{\TO\T{-points}\},\db\\
& \#\Ms_3=q^2-1;~ \Ms_4=\{\RC\T{-points}\},~ \Ms_5=\{\IC\T{-points}\};~\#\Ms_4=\#\Ms_5=(q^3-q)/2.
\end{align*}
\end{theorem}

\begin{theorem}\label{th2:LineClasses}
\emph{\cite[pp.\,2-3, Lemma 2.1]{CerPavDM}, \cite[Chapter 21]{Hirs_PG3q}}
Let $q\ge5$.
  The lines of $\PG(3,q)$ can be partitioned into classes called $\Oc_i$ and $\Oc'_i$, each of which is a union of orbits under~$G_q$.

\textbf{\emph{(i)}} $q\not\equiv0\pmod3.~ Then~ \Oc'_i=\Oc_i\Ak,\,\#\Oc'_i=\#\Oc_i,\,i=1,\ldots,6;~\Oc'_i=\Oc_i,\,i=2,4,6$.
\begin{align*}
  &\Oc_1=\{\RC\T{-lines}\},\Oc'_1=\{\RA\T{-lines}\},\#\Oc_1=(q^2+q)/2;\,\Oc_2=\{\Tr\T{-lines}\},\#\Oc_2=q+1;\db\\
  &\Oc_3=\{\IC\T{-lines}\},\Oc'_3=\{\IA\T{-lines}\},\#\Oc_3=(q^2-q)/2;~\Oc_4=\{\UG\T{-lines}\},\#\Oc_4=q^2+q;\db\\
  &\Oc_5=\{\UnG\T{-lines}\},\Oc'_5=\{\EG\T{-lines}\},\#\Oc_5=q^3-q;\db\\
  &\Oc_6=\{\EnG\T{-lines}\},~\#\Oc_6=(q^2-q)(q^2-1).
\end{align*}

\textbf{\emph{(ii)}} $q\equiv0\pmod3,~q>3$. Then the classes $\Oc_1,\ldots,\Oc_6$ are as in the case (i); also, $\Oc_7=\{\Ar\T{-line}\},~\#\Oc_7=1$; $\Oc_8=\{\EA\T{-lines}\},~\#\Oc_8=(q+1)(q^2-1)$.
\end{theorem}


 \begin{notation}\label{notation2:2}
 Let $q\equiv\xi\pmod3$, $\xi\in\{-1,0,1\}$. Let $W$ be a group of projectivities in $\PG(3,q)$ isomorphic  to $PGL(2,q)$, fixing $\Cs$ of \eqref{eq2:cubic}, and, let the matrix, corresponding to a projectivity of $W$, have the general form \eqref{eq2:q ge5 matrixGq}. For $q\ge5$, we have $W=G_q$; for $q=2,3,4$, $W$ is a subgroup of $G_q$, see below Sections \ref{sec4:q=2}, \ref{sec5:q=3}, \ref{sec6:q=4}. If for $\lambda$-lines, $\lambda\in\Lk^{(\xi)}$, there are $\ge2$ orbits under $W$, then a line of the type $\lambda$, belonging to the $u$-th orbit, is denoted as a line of the type $\lambda_u$. For instance, see in Theorem \ref{th2:line orbits q>=5}, lines of the types $\UG_u,u=1,2$, and $\EnG_u,u=1,\ldots,2q-2+\xi$, for $q=2^n\ge8$; $\UnG_u,u=1,2$, $\EA_u,u=1,2,3$, and $\EnG_u,u=1,\ldots,2q-3$, for $q=3^n\ge9$.
 \end{notation}


\begin{theorem}\label{th2:line orbits q>=5}
 \emph{\cite{BlokPelSzo,CerPavDM}, \cite[Table 1]{DMP_PointLineInc}, \cite[Section 3, Theorem 8.1, Conjecture 8.2]{DMP_OrbLineMJOM}, \cite{DMP_OrbLineO6MJOM,DMP_MJOM2025}, \cite[Chapter 21]{Hirs_PG3q}, \cite{KaPrad-q3orbits}} Let $q\ge5$, $q\equiv\xi\pmod 3$.

 \textbf{\emph{(i)}} Let $q=2^n$, $n\ge3$. Then $\xi\in\{-1,1\}$.

 Under $G_q$, we have one orbit for all $\lambda$-lines with $\lambda\in\Lk^{(\ne0)}\setminus\{\UG,\EnG\}$, two orbits for $\UG$-lines, and $2q-2+\xi$ orbits for $\EnG$-lines.
 In total,  there are the following $2q +7+\xi$ orbits $\Lc_i$ of lines:
 \begin{align*}
& \Lc_1=\{\RC\T{-lines}\},~ \Lc_2=\{\RA\T{-lines}\},~ \#\Lc_1=\#\Lc_2=(q^2+q)/2;~
\Lc_3=\{\Tr\T{-lines}\},\db \\
&\#\Lc_3=q+1;~\Lc_4=\{\IC\T{-lines}\},~ \Lc_5=\{\IA\T{-lines}\},~ \#\Lc_4=\#\Lc_5=(q^2-q)/2;\db\\
&\Lc_6=\{\UG_1\T{-lines}\},~ \#\Lc_6=q+1,~ \Lc_7=\{\UG_2\T{-lines}\},~ \#\Lc_7=q^2-1;\db\\
&\Lc_8=\{\UnG\T{-lines}\},~ \Lc_9=\{\EG\T{-lines}\},~ \#\Lc_8=\#\Lc_9=q^3-q;\db\\
&\Lc_{9+k}=\{\EnG_k\T{-lines}\},~ \#\Lc_{9+k}=(q^3-q)/(2+\xi),~ k=1,\ldots,2+\xi;\db\\
&\Lc_{11+\xi+j}=\{\EnG_{2+\xi+j}\T{-lines}\},~ \#\Lc_{11+\xi+j}=(q^3-q)/2,~ j=1,\ldots,2q-4.
 \end{align*}

 \textbf{\emph{(ii)}} Let $q=3^n$, $n\ge2$.

 Under $G_q$, we have one orbit for all $\lambda$-lines with  $\lambda\in\Lk^{(0)}\setminus\{\UnG,\EA,\EnG\}$, two orbits for $\UnG$-lines, three orbits for $\EA$-lines, and $2q-3$ orbits for  $\EnG$-lines. In total, there are the following $2q +7$ orbits $\Lc_i$ of lines:
 \begin{align*}
&\Lc_1=\{\RC\T{-lines}\},~ \#\Lc_1=(q^2+q)/2;~ \Lc_2=\{\Tr\T{-lines}\},~ \#\Lc_2=q+1;\db\\
&\Lc_3=\{\IC\T{-lines}\},~ \#\Lc_3=(q^2-q)/2; ~\Lc_4=\{\UG\T{-lines}\},~ \#\Lc_4=q^2+q;\db\\
&\Lc_5=\{\UnG_1\T{-lines}\},~ \Lc_6=\{\UnG_2\T{-lines}\},~ \#\Lc_5=\#\Lc_6=(q^3-q)/2;\db\\
&\Lc_7=\{\Ar\T{-lines}\},~ \#\Lc_7=1;~ \Lc_8=\{\EA_1\T{-lines}\},~ \#\Lc_8=q^3-q;\db\\
&\Lc_9=\{\EA_2\T{-lines}\},~ \Lc_{10}=\{\EA_3\T{-lines}\},~ \#\Lc_9=\#\Lc_{10}=(q^2-1)/2;\db\\
&\Lc_{10+k}=\{\EnG_k\T{-lines}\},~ \#\Lc_{10+k}=q^3-q,~ k=1,\ldots,q/3;\db\\
&\Lc_{10+q/3+j}=\{\EnG_{q/3+j}\T{-lines}\},~ \#\Lc_{10+q/3+j}=(q^3-q)/2,~ j=1,\ldots,q-1;\db\\
&\Lc_{9+4q/3+i}=\{\EnG_{4q/3-1+i}\T{-lines}\},~ \#\Lc_{9+4q/3+i}=(q^3-q)/4,~ i=1,\ldots,(2q-6)/3.
\end{align*}
\end{theorem}

\begin{theorem}\label{th2:Gq q=234}
\emph{\cite[Chapter 21, Lemma  21.1.3, Appendix III]{Hirs_PG3q}, \cite[Section 6.4]{Hirs_PGFF}}

For $q=2,3,4$, there are the following group $G_q$ of the of projectivities in $\PG(3,q)$, fixing the twisted cubic $\Cs$ of \eqref{eq2:cubic}:
\begin{align*}
& G_2\cong\mathbf{S}_3\mathbf{Z}_2^3,~ \#G_2=8\cdot\#PGL(2,2)=48;\db \\
& G_3\cong\mathbf{S}_4\mathbf{Z}_2^3,~ \#G_3=8\cdot\#PGL(2,3)=192;\db \\
&G_4\cong\mathbf{S}_5\cong P\Gamma L(2,4)\cong\mathbf{Z}_2PGL(2,4),~ \#G_4=2\cdot\#PGL(2,4)=120.
\end{align*}
 Also, $\mathbf{S}_3\cong PGL(2,2)\cong PSL(2,2)$, $\mathbf{S}_4\cong PGL(2,3)$.
\end{theorem}

\begin{remark}\label{rem2:KnownIncidMatr}
  For $q=2^n\ge8$ and $q=3^n\ge9$, parameters of the incidence matrices $\Ic^{\Pi\Pb}$, $\Ic^{\Lambda\Pb}$, $\Ic^{\Lambda\Pi}$
  and  their submatrices $\Ic^{\Pi\Pb}_{ij}$, $\Ic^{\Lambda\Pb}_{ij}$, $\Ic^{\Lambda\Pi}_{ij}$ can be found in
 \cite[Tables 1,\,2]{BDMP-TwCub}, \cite[Tables 2--5]{DMP_PointLineInc}, \cite[Tables 1,\,2]{DMP_PlLineIncJG}, \cite{CerPavDM,DMP_IncO6JG,DMP_MJOM2025,KaPradInciden}.
\end{remark}

\begin{theorem}\label{th2:plane=point}
\emph{ \cite[Appendix A.2]{Adriaensen}, \cite[Corollary 2.1]{Block1967}}
Let $G\subseteq \mathrm{P}\Gamma \mathrm{L}(k,q)$ be a group of collineations of $\PG(k-1,q)$. In $\PG(k-1,q)$, under $G$, let $\Ns_i$ be the $i$-th hyperplane orbit and $\Ms_j$ be the $j$-th point orbit. Let $\mathbb{N}=\{\Ns_1,\ldots,\Ns_n\}$ be the set of all
hyperplane orbits and $\mathbb{M}=\{\Ms_1,\ldots,\Ms_m\}$ be the set of all point orbits. Then $n=m$, i.e. the numbers of hyperplane orbits and of point orbits are the same. Moreover, $(\mathbb{N},\mathbb{M})$ form a \emph{tactical decomposition} of $\PG(k-1,q)$ that means the following: there exist integers $n_{ij}$ and $m_{ij}$ such that
every hyperplane in $\Ns_i$ is incident with exactly $n_{ij}$ points of $\Ms_j$ and
every point in $\Ms_j$ is incident with exactly $m_{ij}$ hyperplanes of $\Ns_i$.
\end{theorem}

\section{Further notations and useful relations}\label{sed3:useful}
\begin{definition}\label{def3:critic}
 For $q=2,3,4$, a subgroup of $G_q$, say $G_q^*$, is ``\emph{critical}'' if it
 is isomorphic  to $PGL(2,q)$, its projectivities in $\PG(3,q)$ fix $\Cs$ of \eqref{eq2:cubic}, and, finally, as the main property, the matrix, corresponding to a projectivity of $G_q^*$, has the general form \eqref{eq2:q ge5 matrixGq}.
\end{definition}

The following lemma is obvious.

\begin{lemma}\label{lem3:coincid} \emph{\cite[Theorem 6]{DMP_PointLineInc}, \cite[Theorem 3.4]{DMP_PlLineIncJG}, \cite[Theorem 3.3]{DMP_OrbLineMJOM}}
Let $q\equiv\xi\pmod 3$. For $q=2,3,4$ with $\xi=-1,0,1$, respectively,  the orbits of planes, points, and lines under the critical subgroup $G_q^*$ and the corresponding incidence matrices and submatrices  coincide with the $G_q$-orbits and incidence matrices for $q\ge5$ with the same values of $\xi$.
\end{lemma}


\begin{notation}\label{notat3:top-bot entriis}
 The following notations for orbits and incidence matrices and entries in tables are used for $q=2,3,4$.

 --
We mark by the asterisk $^*$ orbits under the critical subgroup $G_q^*$ and the corresponding submatrices of the incidence matrices. For instance, $\Lc_i^*$ is the $i$-th orbit of lines under $G_q^*$; $\Ns_j^*$ is the $j$-th orbit of planes under $G_q^*$; $\Ic^{\Lambda\Pi*}_{ij}$ is an $\#\Lc_i^*\times\#\Ns_j^*$ submatrix of the matrix $\Ic^{\Lambda\Pi}$ in $\PG(3,q)$,
given by the orbits $\Lc_i^*$, $\Ns_j^*$; and so on.

--
In tables below, for the parameters of every submatrix $\Ic_{ij}^{\Pi\Pb}$, $\Ic_{ij}^{\Pi\Pb*}$, $\Ic_{ij}^{\Lambda\Pb}$, $\Ic_{ij}^{\Lambda\Pb*}$, $\Ic_{ij}^{\Lambda\Pb}$, $\Ic_{ij}^{\Lambda\Pi*}$, we assign a cell with two entries: the top entry, say $t^\bullet_{ij}$ or $t^{\bullet*}_{ij}$, and the bottom one, say $b^\bullet_{ij}$ or $b^{\bullet*}_{ij}$, where $\bullet\in\{\Pi\Pb,\Lambda\Pb,\Lambda\Pi\}$ is clear by the context. The top entry is equal to the number of units in every row of the submatrix, whereas the bottom entry is the number of units in every column of the submatrix.
\end{notation}

The following lemma is based on the approaches of \cite[Lemma 4.12, Theorem 6.1, equation (6.1)]{BDMP-TwCub},
\cite[Lemma 1]{DMP_PointLineInc}, \cite[Lemma 4.1]{DMP_PlLineIncJG}, see also the references therein.

\begin{lemma}\label{lem3:top-bottom entries}
\begin{description}
  \item[(i)]
Let an $\#\Ns_i\times\#\Ms_j$ submatrix $\Ic_{ij}^{\Pi\Pb}$ of the $\theta_{3,q}\times\theta_{3,q}$ plane-point incidence matrix $\Ic^{\Pi\Pb}$ be given by the plane orbit $\Ns_i$ and the point orbit $\Ms_j$. Then the number of points from $\Ms_j$ in a plane of $\Ns_i$ is the same for all planes of $\Ns_i$; it is denoted by $t^{\Pi\Pb}_{ij}$ and is used as the top entry in the corresponding cell. Conversely, the number of planes from $\Ns_i$ through a point of $\Ms_j$ is the same for all points of $\Ms_j$; it is denoted by $b^{\Pi\Pb}_{ij}$ and is used as the bottom entry in the cell. Moreover,
\begin{equation}\label{eq3:plane-point}
 t^{\Pi\Pb}_{ij}\cdot\#\Ns_i=b^{\Pi\Pb}_{ij}\cdot\#\Ms_j.
\end{equation}

 For $\Ic_{ij}^{\Pi\Pb*}$, all is similar with change of $\Ns_i,\Ms_j,t^{\Pi\Pb}_{ij},b^{\Pi\Pb}_{ij}$ by $\Ns_i^*,\Ms_j^*,t^{\Pi\Pb*}_{ij},b^{\Pi\Pb*}_{ij}$.
  \item[(ii)]
Let an $\#\Lc_i\times\#\Ms_j$ submatrix $\Ic_{ij}^{\Lambda\Pb}$ of the line-point $\beta_{3,q}\times\theta_{3,q}$ incidence matrix $\Ic^{\Lambda\Pb}$ be given by the line orbit $\Lc_i$ and the point orbit $\Ms_j$. Then the number of points from $\Ms_j$ on a
line of $\Lc_i$ is the same for all lines of $\Lc_i$; it is denoted by $t^{\Lambda\Pb}_{ij}$ and is used as the top entry in the corresponding cell. Conversely, the number of lines from $\Lc_i$ through a point of $\Ms_j$ is the same for all points of $\Ms_j$; it is denoted by $b^{\Lambda\Pb}_{ij}$ and is used as the bottom entry in the cell. Moreover,
\begin{equation}\label{eq3:line-point}
 t^{\Lambda\Pb}_{ij}\cdot\#\Lc_i=b^{\Lambda\Pb}_{ij}\cdot\#\Ms_j.
\end{equation}

 For $\Ic_{ij}^{\Lambda\Pb*}$, all is similar with change of $\Lc_i,\Ms_j,t^{\Lambda\Pb}_{ij},b^{\Lambda\Pb}_{ij}$ by $\Lc_i^*,\Ms_j^*,t^{\Lambda\Pb*}_{ij},b^{\Lambda\Pb*}_{ij}$.

  \item[(iii)] Let an $\#\Lc_i\times\#\Ns_j$ submatrix $\Ic_{ij}^{\Lambda\Pi}$ of the line-plane $\beta_{3,q}\times\theta_{3,q}$ incidence matrix $\Ic^{\Lambda\Pi}$ be given by the line orbit $\Lc_i$ and the plane orbit $\Ns_j$. Then the number of planes from $\Ns_j$ through a
line of $\Lc_i$ is the same for all lines of $\Lc_i$; it is denoted by $t^{\Lambda\Pi}_{ij}$ and is used as the top entry in the corresponding cell. Conversely, the number of lines from $\Lc_i$ in a plane of $\Ns_j$ is the same for all planes of $\Ns_j$; it is denoted by $b^{\Lambda\Pi}_{ij}$ and is used as the bottom entry in the cell. Moreover,
\begin{equation}\label{eq3:line-plane}
 t^{\Lambda\Pi}_{ij}\cdot\#\Lc_i=b^{\Lambda\Pi}_{ij}\cdot\#\Ns_j.
\end{equation}

 For $\Ic_{ij}^{\Lambda\Pi*}$, all is similar with change of $\Lc_i,\Ns_j,t^{\Lambda\Pi}_{ij},b^{\Lambda\Pi}_{ij}$ by $\Lc_i^*,\Ns_j^*,t^{\Lambda\Pi*}_{ij},b^{\Lambda\Pi*}_{ij}$.
\end{description}
\end{lemma}

\begin{proof}
The assertions can be proved similarly to \cite[Lemma 4.12]{BDMP-TwCub},
\cite[Lemma 1, Corollary 2]{DMP_PointLineInc}, \cite[Lemma 4.1, Corollary 4.2]{DMP_PlLineIncJG},\cite[Lemma 3.2]{DMP_IncO6JG}.
\end{proof}

\begin{corollary}\label{cor4:top-bottom entries}
Let $\bullet\in\{\Pi\Pb,\Lambda\Pb,\Lambda\Pi\}$ be clear by the context. For the entries in the submatrices $\Ic_{ij}^{\Pi\Pb}$, $\Ic_{ij}^{\Lambda\Pb}$, $\Ic_{ij}^{\Lambda\Pb}$, the following holds:
\begin{align}
&\T{If }t^\bullet_{ij}=0 \T{ then }b^\bullet_{ij}=0\T{ and vice versa}; \label{eq3:=0}\db\\
&t^{\Pi\Pb}_{ij}=\frac{b^{\Pi\Pb}_{ij}\cdot\#\Ms_j}{\#\Ns_i},
~t^{\Lambda\Pb}_{ij}=\frac{b^{\Lambda\Pb}_{ij}\cdot\#\Ms_j}{\#\Lc_i},
~t^{\Lambda\Pi}_{ij}=\frac{b^{\Lambda\Pi}_{ij}\cdot\#\Ns_j}{\#\Lc_i};\label{eq3:obtain tij}\db\\
&b^{\Pi\Pb}_{ij}=\frac{t^{\Pi\Pb}_{ij}\cdot\#\Ns_i}{\#\Ms_j},
~b^{\Lambda\Pb}_{ij}=\frac{t^{\Lambda\Pb}_{ij}\cdot\#\Lc_i}{\#\Ms_j},
~b^{\Lambda\Pi}_{ij}=\frac{t^{\Lambda\Pi}_{ij}\cdot\#\Lc_i}{\#\Ns_j};\label{eq3:obtain bij}\db\\
&\sum_{j=1}^5 t^{\Pi\Pb}_{ij}= q^2+q+1,~\sum_{j=1}^5 t^{\Lambda\Pb}_{ij}= \sum_{j=1}^5 t^{\Lambda\Pi}_{ij}= q+1,~i\T{ is fixed};\label{eq3:sum j}\db\\
&\sum_{i=1}^5 b^{\Pi\Pb}_{ij}= \sum_{i=1}^{2q+7+\xi} b^{\Lambda\Pb}_{ij}=
\sum_{i=1}^{2q+7+\xi}b^{\Lambda\Pi}_{ij}=q^2+q+1,~j\T{ is fixed}.\label{eq3:sum i}
\end{align}
For the submatrices  $\Ic_{ij}^{\Pi\Pb*}$, $\Ic_{ij}^{\Lambda\Pb*}$, $\Ic_{ij}^{\Lambda\Pi*}$ we have the similar assertions.
\end{corollary}

\begin{proof}
    We use Lemma \ref{lem3:top-bottom entries}. The assertions \eqref{eq3:=0}--\eqref{eq3:obtain bij} follow from \eqref{eq3:plane-point}--\eqref{eq3:line-plane}. The relation \eqref{eq3:sum j} holds as, in $\PG(3,q)$, there are $q^2+q+1$ points in every plane, $q+1$ points on every line, and, finally, $q+1$ planes through every line. The assertion \eqref{eq3:sum i} holds since in $\PG(3,q)$ there are $q^2+q+1$ planes and lines through every point and $q^2+q+1$ lines in every plane.
\end{proof}

\begin{notation}\label{notat3:2}
As an addition to Notations \ref{notation2:1}, \ref{notation2:2}, \ref{notat3:top-bot entriis}, and notations from Theorems \ref{th2:Hirs}, \ref{th2:line orbits q>=5}, and Lemma \ref{lem3:top-bottom entries}, in tables of Sections \ref{sec4:q=2}--\ref{sec6:q=4} we use the following notations.

\textbf{(i)} In the \emph{left part of the tables} (columns $1,\ldots,6$) the orbits under $G_q^*$ are used.

In the 1-st column, the entry of the form
$\begin{array}
{c}
\Ns_i^*:\#\Ns_i^*\\
\pi
\end{array}
$,
notes that $\#\Ns_i^*$ is the number of planes in the orbit $\Ns_i^*$ and $\pi\in\Pk$ is the type of the planes in $\Ns_i^*$.
Similarly, the entry of the form
$\begin{array}
{c}
\Lc_i^*:\#\Lc_i^*\\
\lambda
\end{array}
$,
notes that $\#\Lc_i^*$ is the number of lines in the orbit $\Lc_i^*$ and $\lambda\in\Lk^{(\xi)}$ is the type of the lines in $\Lc_i^*$.

In the top of the columns $2,\ldots,6$, the entry of the form
$\begin{array}
{c}
\Ms_j^*\\
\pk\\
\#\Ms_j^*
\end{array}
$,
notes that $\pk\in\Mk^{(\xi)}$ is the type of the points in the orbit $\Ms_j^*$ and $\#\Ms_j^*$ is the number of points in $\Ms_j^*$.
Similarly, the entry of the form
$\begin{array}
{c}
\Ns_j^*\\
\pi\\
\#\Ns_j^*
\end{array}
$,
notes that $\pi\in\Pk$ is the type of the planes in the orbit $\Ns_j^*$ and $\#\Ns_j^*$ is the number of planes in $\Ns_j^*$.

\textbf{(ii)} In the \emph{right part of the tables} (columns $7,\ldots,11$ or $7,\ldots,12$) the orbits under $G_q$ are used.
At most, the notations are similar to the left part. Many orbits from the right part are unions of orbits from the left one. In this case, the united orbit contains a few types of planes, or points, or lines. In the tables, these types are separated by commas.
For $q=4$, in one of the three tables the right part is absent.
\end{notation}


\section{Orbits and incidence matrices for $q=2$}\label{sec4:q=2}
For $q=2$, the twisted cubic is $\Cs=\{ \Pb(0,0,0,1), \Pb(1,1,1,1), \Pb(1,0,0,0)\}$.
Indeed, in this case $\Cs$ is a conic contained in the plane $\pi_{12}$ of equation $x_1 = x_2$; see  \cite[Chapter 21.1]{Hirs_PG3q}. 
Theorem \ref{th2:Gq q=234} states that $G_2\cong\mathbf{S}_3\mathbf{Z}_2^3$.
 We recall also that in  $\PG(3,2)$ a projective frame has four points. A way to obtain a projective frame is taking the three points of $\Cs$ and one of the eight points not belonging to $\pi_{12}$. This explains the structure of  $G_2\cong\mathbf{S}_3\mathbf{Z}_2^3$. The $\mathbf{S}_3$ part corresponds to the possibility of permuting the three points of $\Cs$, the $\mathbf{Z}_2^3$ part corresponds to the possibility of swapping two points not belonging to $\pi_{12}$ fixing punctually $\Cs$. A matrix corresponding to a projectivity of $G_2$ is the product of a matrix $\Mb$ of the form \eqref{eq2:q ge5 matrixGq} and a matrix  $\Mb'$  of the form
  \begin{equation}\label{mat mprime q2}
 \Mb'=\left[
 \begin{array}{cccc}
 1&0&0&0\\
 e&f&g&h\\
 e&g&f&h\\
 0&0&0&1
 \end{array}
  \right],~e,f,g,h\in\F_2,~ f\ne g.
  \end{equation}

The projectivity corresponding to a matrix $\Mb'$ fixes the twisted cubic punctually, and therefore fix $\pi_{12}$ punctually, and permutes the points outside $\pi_{12}$.


%

 \begin{theorem}\label{th4:Gq*q=2}
Let $q=2$, $q\equiv-1\pmod 3$.

\textbf{\emph{(i)}}
  The group $G_2\cong\mathbf{S}_3\mathbf{Z}_2^3$ contains $8$ subgroups isomorphic to $PGL(2,2)$ divided into two conjugacy classes. For one of these subgroups (and only for it), the matrices corresponding to its projectivities
assume the form described by \eqref{eq2:q ge5 matrixGq}. So, this subgroup (and only it) can be treated as a critical subgroup $G_2^*$.

\textbf{\emph{(ii)}} The plane, point, and line orbits $\Ns^*_i,\Ms^*_i$, and $\Lc^*_i$  under the critical subgroup $G_2^*$ are the same as the orbits $\Ns_i,\Ms_i$, and $\Lc_i$ in Theorems \emph{\ref{th2:Hirs}(ii)(iii)}, \emph{\ref{th2:line orbits q>=5}(i)}. We have,
\begin{align*}
&\Ns^*_1=\{\Gamma\T{-planes}\},~\#\Ns^*_1=3;~\Ns^*_{2}=\{2_\Cs\T{-planes}\},~\#\Ns^*_2=6;~\Ns^*_{3}=\{3_\Cs\T{-planes}\},\db\\
&\#\Ns^*_3=1;~ \Ns^*_{4}=\{\overline{1_\Cs}\T{-planes}\},~ \#\Ns^*_4=3;~\Ns^*_{5}=\{0_\Cs\T{-planes}\},~ \#\Ns^*_5=2.\db\\
&\Ms^*_1=\{\Cs\T{-points}\},~ \Ms^*_2=\{\Tr\T{-points}\},~ \Ms^*_3=\{3_\Gamma\T{-points}\},~ \Ms^*_4=\{1_\Gamma\T{-points}\},~\db\\
&\Ms^*_5=\{0_\Gamma\T{-points}\};~ \#\Ms^*_j=\#\Ns^*_j,~ j=1,~\ldots,~5.\\
&\Lc^*_1=\{\RC\T{-lines}\},~\Lc^*_2=\{\RA\T{-lines}\},~
\Lc^*_3=\{\Tr\T{-lines}\},~\#\Lc^*_1=\#\Lc^*_2=\#\Lc^*_3=3;\db \\
&\Lc^*_4=\{\IC\T{-lines}\},~ \Lc^*_5=\{\IA\T{-lines}\},~ \#\Lc^*_4=\#\Lc^*_5=1;\db\\
&\Lc^*_6=\{\UG_1\T{-lines}\},~ \Lc^*_7=\{\UG_2\T{-lines}\},~ \#\Lc^*_6=\#\Lc^*_7=3;~\Lc^*_8=\{\UnG\T{-lines}\},\db\\
&\Lc^*_9=\{\EG\T{-lines}\},~\Lc^*_{10}=\{\EnG\T{-lines}\},~ \#\Lc^*_8=\#\Lc^*_9=\#\Lc^*_{10}=6.
\end{align*}
The other three subgroups of the conjugacy class of $G_2^*$ have 5 point orbits, too.
Three orbits are $\Ms^*_1$, $\Ms^*_3$, $\Ms^*_4$. The 8 points outside the plane $x_1 = x_2$ are arranged into two sets of size two and six that are different from $\Ms^*_2$, $\Ms^*_5$.
The four subgroups of the other conjugacy class have 6 point orbits. In the plane $x_1 = x_2$ the orbits are the same as the other class, outside the plane there are two orbits of size one and two orbits of size three.
\end{theorem}

\begin{proof}
By direct computation.
\end{proof}

\begin{theorem}\label{th4:PlanePointOrbitG2}
Let $q=2$. Under $G_2\cong\mathbf{S}_3\mathbf{Z}_2^3$, there are $4$ plane orbits $\Ns_i$ and $4$ the point orbits $\Ms_j$ such that
\begin{align}
&\Ns_1=\Ns^*_1\cup\Ns^*_4,~\Ns_{2}=\Ns^*_{2},~\Ns_{3}=\Ns^*_{3},~\Ns_{4}=\Ns^*_{5}.\label{eq4:PlaneOrbitG2-1}\db\\
&\Ms_1=\Ms^*_1,~ \Ms_2=\Ms^*_2\cup\Ms^*_5,~\Ms_3=\Ms^*_3~ \Ms_4=\Ms^*_4.\label{eq4:PointOrbitG2-1}
\end{align}
\end{theorem}

\begin{proof}
The orbit $\Ms^*_4$ consists of the  three points on $\RC\T{-lines}$ not belonging $\Cs$. The orbit $\Ms^*_3$ consists of the meeting point of the three $\UG\T{-lines}$ contained in $\pi_{12}$. As the matrices of the form \eqref{eq2:q ge5 matrixGq} fix the points of $\pi_{12}$, $\Ms^*_3 = \Ms_3$, $\Ms^*_4 = \Ms_4$.
Tangent lines lie outside $\pi_{12}$, otherwise two tangent lines would meet. The  $0_\Gamma\T{-points}$ do not belong to $\pi_{12}$, too. Therefore there exist a matrix of the form \eqref{mat mprime q2} mapping a point of $\Ms^*_2$ into a point of $\Ms^*_5$.
The matrix $\Mb'_1$ of the form \eqref{mat mprime q2} with $e = f = 0$, $g = h = 1$, maps the $\Gamma\T{-plane}$ $x_3=0$ into the $\overline{1_\Cs}\T{-plane}$ $x_1 + x_2 = -x_3$.
By Theorem \ref{th2:plane=point},  $G_2$ has four orbits of planes, too.
\end{proof}

\begin{corollary}
Let $q=2$. The plane and point orbits under $G_2\cong\mathbf{S}_3\mathbf{Z}_2^3$ are as follows.
\begin{align}
&\Ns_1=\{\Gamma\T{-planes},\overline{1_\Cs}\T{-planes}\},~\#\Ns_1=3+3=6;~
\Ns_{2}=\{2_\Cs\T{-planes}\},~\#\Ns_2=6;\label{eq4:PlaneOrbitG2-2}\db\\
&~\Ns_{3}=\{3_\Cs\T{-planes}\},~\#\Ns_3=1;~\Ns_{4}=\{0_\Cs\T{-planes}\},~ \#\Ns_4=2.\notag\db\\
&\Ms_1=\{\Cs\T{-points}\};~ \Ms_2=\{\Tr\T{-points},0_\Gamma\T{-points}\},~\#\Ms_2=6+2=8;\label{eq4:PointOrbitG2-2}\db\\
&\Ms_3=\{3_\Gamma\T{-points}\},~\#\Ms_3=1;~ \Ms_4=\{1_\Gamma\T{-points}\},~\#\Ms_1=\#\Ms_4=3.\notag
\end{align}
\end{corollary}

\begin{proof}
The assertions follow from Theorems \ref{th4:Gq*q=2}, \ref{th4:PlanePointOrbitG2}.
\end{proof}


\begin{theorem}\label{th4:LineOrbitG2}
Let $q=2$. Under $G_2\cong\mathbf{S}_3\mathbf{Z}_2^3$, there are $6$ line orbits $\Lc_i$ such that
\begin{align*}
&\Lc_1=\Lc^*_1, ~\Lc_2=\Lc^*_5, ~\Lc_3=\Lc^*_7, ~\Lc_4=\Lc^*_2  \cup \Lc^*_4, \db\\
&\Lc_5=\Lc^*_3 \cup \Lc^*_6 \cup \Lc^*_8, ~\Lc_6=\Lc^*_9 \cup \Lc^*_{10}.
\end{align*}

\end{theorem}

\begin{proof}
There are three  $\UG$-lines contained in the plane  $\pi_{12}$. They are permuted by the matrices of type of the form \eqref{eq2:q ge5 matrixGq} and are fixed punctually by the matrices of the form \eqref{mat mprime q2}, so $\Lc_2=\Lc^*_5$. We call this orbit $\UG_2$. Also $\RC$-lines and the $\IA$-line that has equations $x_1 = x_2$, $x_0 = x_2 + x_3$ lie in the plane  $\pi_{12}$, so they are fixed punctually by the matrices of the form \eqref{mat mprime q2}.
The matrix   $ \Mb'_1$ of the proof of Theorem \ref{th4:PlanePointOrbitG2}
maps the $\RA$-line $x_0+x_1 = x_2$, $x_3 = 0$ into the $\IC$-line $x_0 = x_3 $, $x_0+x_1 = x_2$,
maps the $\Tr$-line $x_0= 0$, $x_1 = 0$ into the $\UG_1$-line $x_0 = 0 $, $x_2 = 0$,
maps the $\Tr$-line $x_2= 0$, $x_3 = 0$ into the $\UnG$-line $x_1 = 0 $, $x_2 = x_3$,
maps the $\EG$-line $x_2= 0$, $x_0+x_1 =x_3$ into the $\EnG$-line $x_0 = x_3$, $x_1 = 0 $.
By direct computation, no other line orbits merge.
\end{proof}

\begin{corollary}
Let $q=2$. The line orbits under $G_2\cong\mathbf{S}_3\mathbf{Z}_2^3$ are as follows.
\begin{align}
&\Lc_1=\{\RC\T{-lines}\},  ~\#\Lc_1=3; ~~\Lc_2=\{\IA\T{-lines}\},~\#\Lc_2=1; \db \\
&\Lc_3=\{\UG_2\T{-lines}\},~\#\Lc_3=3;
~~\Lc_4=\{\RA\T{-lines}\} \cup \{\IC\T{-lines}\},~ \#\Lc_4=3+1=4;
\db\\
& \Lc_5=\{\Tr\T{-lines}\} \cup  \{\UG_1\T{-lines}\}   \cup  \{\UnG\T{-lines}\} ,~\#\Lc_5=3+3+6=12;\db\\
& \Lc_6=\{\EG\T{-lines}\} \cup  \{\EnG\T{-lines}\}, ~\#\Lc_6=6+6=12.\notag
\end{align}
\end{corollary}

\begin{proof}
The assertions follow from Theorems \ref{th4:Gq*q=2}, \ref{th4:LineOrbitG2}.
\end{proof}

The above results for $q=2$ are illustrated by Tables \ref{tab4:plane-point q=2}--\ref{tab4:line-plane q=2}.

The \emph{left parts} (columns $1,\ldots,6$) of Tables \ref{tab4:plane-point q=2}--\ref{tab4:line-plane q=2} are as follows.
In the column $1$ and the top of the columns $2,\ldots,6$, see Notation \ref{notat3:2}(i), the entries correspond to Theorem \ref{th4:Gq*q=2}(ii). The order of orbits in the left parts is related to the union of the orbits under $G_2^*$ into the orbits under $G_2$.

In Tables 4.1-4.3 (and in the tables in the next sections), the top and bottom entries in every cell of the left parts, see Notation \ref{notat3:top-bot entriis} and Lemma \ref{lem3:top-bottom entries}, correspond to the known results in accordance with Remark \ref{rem2:KnownIncidMatr} and Lemma \ref{lem3:coincid}. We have checked these entries with the help of the relations \eqref{eq3:=0}--\eqref{eq3:sum i} and the system Magma \cite{Magma}.

Also, in all the tables, the notations $\widebar{a}$, $\widetilde{b}$, $\widehat{c}$, are connected with the following obvious necessary conditions for the union of the orbits, respectively: to merge two or more orbits described in consecutive rows of the tables, the first value of the cells, noted by $\widebar{a}$, must be the same for each orbit. Moreover, these values must be the same in the new merged orbit. To merge two or more orbits described in consecutive columns of the tables, the second value of the cells, noted by $\widetilde{b}$, must be the same for each orbit. Moreover, these values must be the same in the new merged orbit. To merge two or more orbits described in consecutive rows and columns of the tables, the sum of the first values of the cells, noted by $\widehat{c}$,  must be the same for all merged rows and the sum of the second values of the cells, noted by $\widehat{c}$  too, must be the same for all merged columns. Moreover, these sums must be equal to the corresponding values in the new merged orbit.

\begin{table}[htbp]
\caption{Parameters of $\#\Ns_i^*\times\#\Ms_j^*$ submatrices $\Ic_{ij}^{\Pi\Pb*}$ and $\#\Ns_i\times\#\Ms_j$ submatrices $\Ic_{ij}^{\Pi\Pb}$ of the $15\times 15$ plane-point incidence matrix $\Ic^{\Pi\Pb}$ in $\PG(3,q)$ for $q=2$ and orbits of planes $\Ns_i^*,\Ns_i$ and of points $\Ms_j^*,\Ms_j$ under $G_2^*$, $G_2$, respectively}
 \centering
  \begin{tabular}
  {c||c|c|c||c|c||c||c|c|c||c}\hline
\multicolumn{6}{c||}{orbits under $G_2^*\cong PGL(2,2)$}&\multicolumn{5}{c}{orbits under $G_2\cong\mathbf{S}_3\mathbf{Z}_2^3$}\\\hline
$\Ns_i^*:\#\Ns_i^*$&$\Ms_1^*$&$\Ms_3^*$&$\Ms_4^*$&$\Ms_2^*$&$\Ms_5^*$&$\Ns_i:\#\Ns_i$&$\Ms_1$&$\Ms_3$&$\Ms_4$&$\Ms_2$\\
$\pi$&$\Cs$&$3_\Gamma$&$1_\Gamma$&$\Tr$&$0_\Gamma$&$\pi$&$\Cs$&$3_\Gamma$&$1_\Gamma$&$\Tr, 0_\Gamma$\\
$\downarrow$&3&1&3&6&2&$\downarrow$&3&1&3&8\\\hline\hline

$\Ns_2^*: 6$&2&0&1&3&1&$\Ns_2: 6$&2&0&1&4\\
$2_\Cs$   &4&0&2&\tthr&\tthr&$2_\Cs$&4&0&2&\tthr\\\hline

$\Ns_3^*: 1$&3&1&3&0&0&$\Ns_3: 1$&3&1&3&0\\
$3_\Cs$   &1&1&1&\tz&\tz &$3_\Cs$&1&1&1&\tz \\\hline

$\Ns_5^*: 2$&0&0&3&3&1&$\Ns_4: 2$&0&0&3&4\\
$0_\Cs$&0&0&2&\tone&\tone   &$0_\Cs$&0&0&2&\tone\\\hline\hline\vphantom{$H^{H^{H^H}}$}

$\Ns_1^*: 3$   &\beo&\beo&\beo&\wf&\wz&$\Ns_1: 6$&\beo&\beo&\beo&\wf\\
$\Gamma$&1&3&1&\wt&\wz&$\Gamma,\overline{1_\Cs}$&2&6&2&\wthr\\\hline\vphantom{$H^{H^{H^H}}$}

$\Ns_4^*:3$&\beo&\beo&\beo&\wt&\wt&$3+3=6$&\multicolumn{3}{r}{ \wf\pl\wz\el\wt\pl\wt\el\wf}&$\uparrow$\\
$\overline{1_\Cs}$&1&3&1&\wo&\wthr    &&\multicolumn{3}{r}{\wt\pl\wo\el\wz\pl\wthr\el\wthr}&\\\hline
\end{tabular}
\label{tab4:plane-point q=2}
\end{table}

\begin{table}[htbp]
\caption{Parameters of $\#\Lc_i^*\times\#\Ms_j^*$ submatrices $\Ic_{ij}^{\Lambda\Pb*}$ and $\#\Lc_i\times\#\Ms_j$ submatrices $\Ic_{ij}^{\Lambda\Pb}$ of the $35\times 15$ line-point incidence matrix $\Ic^{\Lambda\Pb}$ in $\PG(3,q)$ for $q=2$ and orbits of lines $\Lc_i^*,\Lc_i$ and of points $\Ms_j^*,\Ms_j$ under $G_2^*$, $G_2$, respectively}
 \centering
\begin{tabular} {c||c|c|c||c|c||c||c|c|c||c}\hline
\multicolumn{6}{c||}{orbits under $G_2^*\cong PGL(2,2)$}&\multicolumn{5}{c}{orbits under $G_2\cong\mathbf{S}_3\mathbf{Z}_2^3$}\\\hline
$\Lc^*_i:\#\Lc^*_i$&$\Ms^*_1$&$\Ms^*_3$&$\Ms^*_4$&$\Ms^*_2$&$\Ms^*_5$&$\Lc_i:\#\Lc_i$&$\Ms_1$&$\Ms_3$&$\Ms_4$&$\Ms_2$\\
$\lambda$ &$\Cs$&$3_\Gamma$&$1_\Gamma$&$\Tr$&$0_\Gamma$&$\lambda$&$\Cs$&$3_\Gamma$&$1_\Gamma$&$\Tr,0_\Gamma$\\
$\downarrow$ &3&1&3&6&2  &$\downarrow$ &3&1&3&8\\\hline\hline

$\Lc^*_1: 3$ &2&0&1&0   &0   &$\Lc_1:3$             &2&0&1&0 \\
$\RC$        &2&0&1&\tz&\tz&$\RC$                 &2&0&1&\tz\\\hline

$\Lc^*_5: 1$ &0&0&3&0   &0   &$\Lc_2:1$             &0&0&3&0 \\
$\IA$        &0&0&1&\tz&\tz&$\IA$                 &0&0&1&\tz\\\hline

$\Lc^*_7: 3$ &1&1&1&0   &0   &$\Lc_3:3$             &1&1&1&0 \\
$\UG_2$      &1&3&1&\tz&\tz&$\UG_2$               &1&3&1&\tz\\\hline\hline

$\Lc^*_2: 3 $ &\bez&\beo&\bez&\wt $\vphantom{H^{H^{H^H}}}$ &\wz&$\Lc_4:4$&\bez&\beo&\bez&\wt \\
$\RA$        &0&3&0&\wo&\wz     &$\RA,\IC$          &0&4&0&\wo   \\\hline

$\Lc^*_4: 1$ &\bez&\beo&\bez& \wz &\wt $\vphantom{H^{H^{H^H}}}$ &$3+1=4$&\multicolumn{3}{r}{\wt\pl\wz\el\wz\pl\wt\el\wt} &$\uparrow$ \\
$\IC$        &0&1&0&\wz&\wo&&\multicolumn{3}{r}{\wo\pl\wz\el\wz\pl\wo\el\wo} &  \\\hline\hline

$\Lc^*_3: 3$ &\beo&\bez&\bez&\wt $\vphantom{H^{H^{H^H}}}$  &\wz &$\Lc_5:12$&\beo&\bez&\bez&\wt\\
$\Tr$      &1   &0   &0   &\wo&\wz &$\Tr,\UG_1,\UnG$&4&0&0&\wthr\\\hline

$\Lc^*_6: 3$ &\beo&\bez&\bez&\wt $\vphantom{H^{H^{H^H}}}$ &\wz &$3+3+$&\multicolumn{3}{r}{}& $\uparrow$ \\
$\UG_1$      &1&0&0&\wo&\wz&$6=12$& \multicolumn{4}{r}{\wt\pl\wz\el\wt\pl\wz\el\wo\pl\wo\el\wt}\\\cline{1-6}

$\Lc^*_8: 6$ &\beo&\bez&\bez&\wo $\vphantom{H^{H^{H^H}}}$   &\wo&&\multicolumn{4}{r}{\wo\pl\wo\pl\wo\el\wz\pl\wz\pl\wthr\el\wthr}\\
$\UnG$       &2   &0   &0   &\wo&\wthr&\\\hline\hline

$\Lc^*_9: 6$ &\bez&\bez&\beo&\wt $\vphantom{H^{H^{H^H}}}$ &\wz &$\Lc_6:12$            &\bez&\bez&\beo&\wt \\
$\EG$        &0&0&2&\wt&\wz&$\EG,\EnG$       &0&0&4&\wthr   \\\hline

$\Lc^*_{10}:6$&\bez&\bez&\beo&\wo $\vphantom{H^{H^{H^H}}}$   &\wo&$6+6=12$&\multicolumn{3}{r}{\wt\pl\wz\el\wo\pl\wo\el\wt} &$\uparrow$ \\
$\EnG$      &0&0&2&\wo&\wthr&&\multicolumn{3}{r}{\wt\pl\wo\el\wz\pl\wthr\el\wthr} &\\\hline
\end{tabular}
\label{tab4:line-point q=2}
\end{table}

\begin{table}[htbp]
\caption{Parameters of $\#\Lc_i^*\times\#\Ns_j^*$ submatrices $\Ic_{ij}^{\Lambda\Pi*}$ and $\#\Lc_i\times\#\Ns_j$ submatrices $\Ic_{ij}^{\Lambda\Pi}$ of the $35\times 15$ line-plane incidence matrix $\Ic^{\Lambda\Pi}$ in $\PG(3,q)$ for $q=2$ and orbits of lines $\Lc_i^*,\Lc_i$ and of planes $\Ns_j^*,\Ns_j$ under $G_2^*$, $G_2$, respectively}
 \centering
  \begin{tabular} {c||c|c|c||c|c||c||c|c|c||c}\hline
\multicolumn{6}{c||}{orbits under $G_2^*\cong PGL(2,2)$}&\multicolumn{5}{c}{orbits under $G_2\cong\mathbf{S}_3\mathbf{Z}_2^3$}\\\hline
$\Lc^*_i:\#\Lc^*_i$&$\Ns^*_2$&$\Ns^*_3$&$\Ns^*_5$&$\Ns^*_1$&$\Ns^*_4$&$\Lc_i:\#\Lc_i$&$\Ns_2$&$\Ns_3$&$\Ns_5$&$\Ns_1$\\
$\lambda$ &$2_\Cs$&$3_\Cs$&$0_\Cs$&$\Gamma$&$\overline{1_\Cs}$&$\lambda$ &$2_\Cs$&$3_\Cs$&$0_\Cs$&$\Gamma,\overline{1_\Cs}$\\
$\downarrow$ &6&1&2&3&3&$\downarrow$ &6&1&2&6\\\hline\hline

$\Lc^*_1: 3$   &2&1&0&0&0&$\Lc_1:3$&2&1&0&0 \\
$\RC$          &1&3&0&\tz&\tz&$\RC$    &1&3&0&\tz \\\hline

$\Lc^*_5: 1$   &0&1&2&0&0&$\Lc_2:1$&0&1&2&0 \\
$\IA$          &0&1&1&\tz&\tz&$\IA$    &0&1&1&\tz\\\hline

$\Lc^*_7: 3$   &0&1&0&1&1&$\Lc_3:3$&0&1&0&2 \\
$\UG_2$        &0&3&0&\tone&\tone&$\UG_2$  &0&3&0&\tone \\\hline\hline

$\Lc^*_2: 3$   &\bez&\bez&\bez&\wt $\vphantom{H^{H^{H^H}}}$&\wo&$\Lc_4:4$   &\bez&\bez&\bez&\wthr \\
$\RA$          &0&0&0&\wt&\wo&$\RA,\IC$&0&0&0&\wt \\\hline

$\Lc^*_4: 1$   &\bez&\bez&\bez&\wz $\vphantom{H^{H^{H^H}}}$ &\wthr&$3+1=4$&\multicolumn{3}{r}{\wt\pl\wo\el\wz\pl\wthr\el\wthr}&$\uparrow$\\
$\IC$          &0&0&0&\wz&\wo&&\multicolumn{3}{r}{\wt\pl\wz\el\wo\pl\wo\el\wt}&\\\hline\hline

$\Lc^*_3: 3$   &\bet&\bez&\bez&\wo $\vphantom{H^{H^{H^H}}}$&\wz&$\Lc_5:12$            &\bet&\bez&\bez&\wo\\
$\Tr$          &1&0&0&\wo&\wz     &$\Tr,\UG_1,\UnG$&4&0&0&\wt\\\hline

$\Lc^*_6: 3$   &\bet&\bez&\bez&\wo $\vphantom{H^{H^{H^H}}}$&\wz&$3+3+$&\multicolumn{3}{r}{}&$\uparrow$ \\
$\UG_1$        &1&0&0&\wo&\wz&$6=12$    &\multicolumn{4}{r}{\wo\pl\wz\el\wo\pl\wz\el\wz\pl\wo\el\wo}\\\cline{1-6}

$\Lc^*_8: 6$   &\bet&\bez&\bez&\wz&\wo $\vphantom{H^{H^{H^H}}}$&& \multicolumn{4}{r}{\wo\pl\wo\pl\wz\el\wz\pl\wz\pl\wt\el\wt}\\
$\UnG$         &2&0&0&\wz&\wt&\\\hline\hline

$\Lc^*_9: 6$   &\beo&\bez&\beo&\wo $\vphantom{H^{H^{H^H}}}$&\wz&$\Lc_6:12$     &\beo&\bez&\beo&\wo\\
$\EG$          &1&0&3&\wt&\wz&$\EG,\EnG$&2&0&6&\wt\\\hline

$\Lc^*_{10}:6$ &\beo&\bez&\beo&\wz&\wo $\vphantom{H^{H^{H^H}}}$&$6+6=12$&\multicolumn{3}{r}{\wt\pl\wz\el\wz\pl\wt\el\wt}&$\uparrow$\\
$\EnG$       &1&0&3&\wz&\wt&&\multicolumn{3}{r}{\wo\pl\wz\el\wz\pl\wo\el\wo}&\\\hline
\end{tabular}
\label{tab4:line-plane q=2}
\end{table}

\section{Orbits and incidence matrices for $q=3$}\label{sec5:q=3}

Theorem \ref{th2:Gq q=234} states that $G_3\cong\mathbf{S}_4\mathbf{Z}_2^3$. As for the case $q = 2$, a way to obtain a
projective frame in $\PG(3,3)$ is taking the points of the twisted cubic $\Cs$ and one in general position with respect to three points of $\Cs$.

Let $\pi_{ijk}$ be the plane through the points $P(i), P(j),P(k)$, $i,j,k\in\{-1,0,1,\infty\}$ of the twisted cubic.
There are
\begin{align*}
&\#\PG(3,3)-\#\pi_{-101}-(\#\pi_{-10\infty}-\#(\pi_{-10\infty}\cap\pi_{-101}))\db\\
&-(\#\pi_{-11\infty}-\#\{P(\infty)\}-\#(\pi_{-11\infty}\cap\pi_{-101})-(\#(\pi_{-11\infty}\cap\pi_{-10\infty})-\#\{ P(-1),P(\infty) \}))\db\\
&-(\#\pi_{p1\infty}-\#\{P(\infty)\}-\#(\pi_{01\infty}\cap\pi_{-101})-(\#(\pi_{01\infty}\cap\pi_{-10\infty})-\#\{P(0),P(\infty)\})\db\\
&   -(\#(\pi_{01\infty}\cap\pi_{-11\infty}) - \#\{ P(1),P(\infty) \}))\db\\
&=40- 13 - (13-4) - (13 - 1 - 4 - (4-2)) - (13 - 1 - 4 -(4-2) -(4-2)) =8
\end{align*}
such points in general position. Let $\Ic$ be the set of this $8$ points.
 This explains the structure of $G_3\cong\mathbf{S}_4\mathbf{Z}_2^3$. The $\mathbf{S}_4$ part corresponds to the
possibility of permuting the four points of $\Cs$, the $\mathbf{Z}_2^3$
part corresponds to the possibility
of swapping two points of $\Cs$, fixing $\Ic$ punctually.
A matrix corresponding to a projectivity of $G_3$ is the product of a matrix $\Mb$ of the form \ref{eq2:q ge5 matrixGq} and a matrix  $\Mb''$  of the form
  \begin{equation}\label{mat mprime q3}
 \Mb''=\left[
 \begin{array}{cccc}
 1&0&0&0\\
 e&f+1&e&g\\
 f&e&f+1&e\\
 0&0&0&1+f-g
 \end{array}
  \right],~e,f,g\in\F_3,~ e \ne \pm(f+1), g \ne f+1.
  \end{equation}

The projectivity corresponding to a matrix $\Mb''$ fixes the twisted cubic punctually,  and permutes the points of $\Ic$.

 \begin{theorem}\label{th5:Gq*q=3}
Let $q=3$.

\textbf{\emph{(i)}}
  The group $G_3\cong\mathbf{S}_4\mathbf{Z}_2^3$ contains $24$ subgroups isomorphic to $PGL(2,3)$ divided into four conjugacy classes.
For one of these subgroups (and only for it), the matrices corresponding to its projectivities
assume the form described by \eqref{eq2:q ge5 matrixGq}. So, this subgroup (and only it) can be treated as a critical subgroup $G_3^*$.

\textbf{\emph{(ii)}} The plane, points, and line orbits $\Ns^*_i,\Ms^*_i$, and $\Lc^*_i$  under the critical subgroup $G_3^*$ are the same as the orbits $\Ns_i,\Ms_i$, and $\Lc_i$ in Theorems \emph{\ref{th2:Hirs}(ii)(iv)},  \emph{\ref{th2:line orbits q>=5}(ii)}. We have,
\begin{align*}
&\Ns^*_1=\{\Gamma\T{-planes}\},~\#\Ns^*_1=4;~\Ns^*_{2}=\{2_\Cs\T{-planes}\},~\#\Ns^*_2=12;~\Ns^*_{3}=\{3_\Cs\T{-planes}\},\db\\
&\#\Ns^*_3=4;~ \Ns^*_{4}=\{\overline{1_\Cs}\T{-planes}\},~ \#\Ns^*_4=12;~\Ns^*_{5}=\{0_\Cs\T{-planes}\},~ \#\Ns^*_5=8.\db\\
& \Ms^*_1=\{\Cs\T{-points}\},\, \Ms^*_2=\{4_\Gamma\T{-points}\},~ \#\Ms^*_1=\#\Ms^*_2=4;~
\Ms^*_3=\{\TO\T{-points}\},\db\\
& \#\Ms^*_3=8;~ \Ms^*_4=\{\RC\T{-points}\},~ \Ms^*_5=\{\IC\T{-points}\};~\#\Ms^*_4=\#\Ms^*_5=12.\\
&\Lc^*_1=\{\RC\T{-lines}\},~ \#\Lc^*_1=6;~ \Lc^*_2=\{\Tr\T{-lines}\},~ \#\Lc^*_2=4; ~\Lc^*_3=\{\IC\T{-lines}\},~ \#\Lc^*_3=3; ~\db\\
&\Lc^*_4=\{\UG\T{-lines}\},~ \#\Lc^*_4=12; ~\Lc^*_5=\{\UnG_1\T{-lines}\},~ \Lc^*_6=\{\UnG_2\T{-lines}\}, \db\\
&\#\Lc^*_5=\#\Lc^*_6=12;~\Lc^*_7=\{\Ar\T{-lines}\},~ \#\Lc^*_7=1;~ \Lc^*_8=\{\EA_1\T{-lines}\},~ \#\Lc^*_8=24;\db\\
&\Lc^*_9=\{\EA_2\T{-lines}\},~ \Lc^*_{10}=\{\EA_3\T{-lines}\},~ \#\Lc^*_9=\#\Lc^*_{10}=4; ~ \Lc^*_{11}=\{\EnG_1\T{-lines}\},\db\\
&\#\Lc^*_{11}=24;~\Lc^*_{12}=\{\EnG_2\T{-lines}\},~ \Lc^*_{13}=\{\EnG_3\T{-lines}\},~ \#\Lc^*_{12}=\#\Lc^*_{13}=12.\db
\end{align*}
The other seven subgroups of the conjugacy class of $G_3^*$ have five point orbits, too. Three orbits are $\Ms^*_1$, $\Ms^*_3$, $\Ms^*_4$. The 16 points of  $\Ms^*_2$, $\Ms^*_5$ are arranged into two sets of size 4 and 12 that are different from $\Ms^*_2$, $\Ms^*_5$.

Another conjugacy class  contains eight subgroups. These subgroups have eight point orbits. One is  $\Ms^*_1$, the points of $\Ms^*_3$ are divided into three orbits of sizes $1,3,12$. The fixed point is different for each subgroup. The remaining points are divided into four orbits of sizes $4,6,6,12$.

The last two conjugacy classes contain four subgroups. In both cases, $\Ms^*_1$ is divided into two orbits of sizes $1$ and $3$; in each class, any subgroup fixes a different point. For the subgroups of a class, $\Ms^*_3$ is an orbit, for the subgroups of the other class, the points of  $\Ms^*_3$ are divided into two orbits of sizes $4$. In both
cases the remaining points are divided into four orbits of sizes $4,6,6,12$.
\end{theorem}

\begin{proof}
By direct computation.

\end{proof}

\begin{theorem}\label{th4:PlanePointOrbitG3}
Let $q=3$. Under $G_3\cong\mathbf{S}_4\mathbf{Z}_2^3$, there are $4$ plane orbits $\Ns_i$ and $4$ the point orbits $\Ms_j$ such that
\begin{align}
&\Ns_1=\Ns^*_1\cup\Ns^*_4,~\Ns_{2}=\Ns^*_{2},~\Ns_{3}=\Ns^*_{3},~\Ns_{4}=\Ns^*_{5}.\label{eq4:PlaneOrbitG2-1}\db\\
&\Ms_1=\Ms^*_1,~ \Ms_2=\Ms^*_2\cup\Ms^*_5,~\Ms_3=\Ms^*_3~ \Ms_4=\Ms^*_4.\label{eq4:PointOrbitG3-1}
\end{align}
\end{theorem}

\begin{proof}
A matrix of the form \eqref{mat mprime q3} fixes $\Cs$ punctually and maps a real chord  into itself. Therefore  $\Ms_1=\Ms^*_1,~ \Ms_4=\Ms^*_4$.
Tangents lie outside 3-secant planes, so by the definition of the matrices of the form \eqref{mat mprime q3} $\Ms_3=\Ms^*_3$.

The matrix $\Mb''_1$ of the form \eqref{mat mprime q3} with $e = f = 0$, $g = -1$, maps
the $4_\Gamma$-point $\Pb(0,1,0,0)$ into the $\IC$-point $\Pb(0,1,0,-1)$.
The matrix $\Mb''_1$ also maps
the $\Gamma\T{-plane}$ $x_3=0$ into the $\overline{1_\Cs}\T{-plane}$ $x_1 = - x_3$.
By Theorem \ref{th2:plane=point}, $G_3$ has four orbits of planes, too.
\end{proof}

\begin{corollary}
Let $q=3$. The plane and point orbits under $G_3\cong\mathbf{S}_4\mathbf{Z}_2^3$ are as follows.
\begin{align}
&\Ns_1=\{\Gamma\T{-planes},\overline{1_\Cs}\T{-planes}\},~\#\Ns_1=4+12=16;~
\Ns_{2}=\{2_\Cs\T{-planes}\},~\#\Ns_2=12;\label{eq4:PlaneOrbitG3-2}\db\\
&\Ns_{3}=\{3_\Cs\T{-planes}\},~\#\Ns_3=4;~\Ns_{4}=\{0_\Cs\T{-planes}\},~ \#\Ns_4=8.\notag\db\\
&\Ms_1=\{\Cs\T{-points}\};~ \Ms_2=\{4_\Gamma\T{-points}, \IC\T{-points}\},~\#\Ms_2=4+12=16;\label{eq4:PointOrbitG3-2}\db\\
&\Ms_3=\{\TO\T{-points}\},~\#\Ms_3=8;~ \Ms_4=\{\RC\T{-points}\},~\#\Ms_4=12.\notag
\end{align}
\end{corollary}

\begin{proof}
The assertions follow from Theorems \ref{th5:Gq*q=3}, \ref{th4:PlanePointOrbitG3}.
\end{proof}

\begin{theorem}\label{th5:LineOrbitG3}
Let $q=3$. Under $G_3\cong\mathbf{S}_4\mathbf{Z}_2^3$, there are $7$ line orbits $\Lc_i$ such that
\begin{align*}
&\Lc_1=\Lc^*_1, ~\Lc_2=\Lc^*_{12}, ~\Lc_3=\Lc^*_2 \cup \Lc^*_6, ~\Lc_4=\Lc^*_3 \cup \Lc^*_7 \cup  \Lc^*_9, \db\\
&\Lc_5=\Lc^*_4 \cup \Lc^*_5, ~\Lc_6=\Lc^*_8 \cup \Lc^*_{11}
~\Lc_7=\Lc^*_{10} \cup \Lc^*_{13}.
\end{align*}

\end{theorem}

\begin{proof}
A matrix of the form \eqref{mat mprime q3} fixes $\Cs$ punctually, so it maps a real chord into itself. Therefore  $\Lc_1=\Lc^*_1$. By direct computation, $\EnG_1$, $\EnG_2$, $\EnG_3$-lines meet $\Ic$ in $1$, $2$, $0$ points, respectively. Therefore a matrix of the form \eqref{mat mprime q3} maps an  $\EnG_2$-line into an  $\EnG_2$-line, so $\Lc_2=\Lc^*_{12}$.
The matrix   $ \Mb''_1$ of the proof of Theorem \ref{th4:PlanePointOrbitG3} maps the $\Tr$-line $x_2 = 0$, $x_3 = 0$ into the $\UnG_2$-line $x_2 = 0$, $x_1 =  -x_3$, maps the $\Ar$-line $x_0 = 0$, $x_3 = 0$ into the $\EA_2$-line $x_0= 0$, $x_1 = - x_3$,
maps the $\UG$-line $x_1 = x_2$, $x_3 = 0$ into the $\UnG_1$-line $x_1 = -x_2$, $x_1 = - x_3$,
maps the $\EA_1$-line $x_2 = 0$, $x_0 = -x_3$, into the $\EnG_1$-line $x_0 = x_1+x_3$, $x_2 = 0$,
maps the $\EA_3$-line $x_0 = x_2$, $x_3 = 0$, into the $\EnG_3$-line $x_0 = x_2$, $x_1= - x_3$.
Finally, the matrix $\Mb''_2$ of the form \eqref{mat mprime q3} with $e=f=-1$, $g=1$, maps the $\Ar$-line $x_0=0$, $x_3=0$ into the $\IC$-line $x_1 +x_2= x_0$, $x_2 +x_3= x_1$.
By direct computation, no other line orbits merge.
\end{proof}

\begin{corollary}
Let $q=3$. The line orbits under $G_3\cong\mathbf{S}_4\mathbf{Z}_2^3$ are as follows.
\begin{align}
& \Lc_1=\{\RC\T{-lines}\}, ~\#\Lc_1=6; ~\Lc_2= \{\EnG_2\T{-lines}\}, ~\#\Lc_2=12;  \notag \db\\
& \Lc_3=\{\Tr\T{-lines}\} \cup  \{\UG_2\T{-lines}\},  ~\#\Lc_3=4+12=16; \notag \db\\
&\Lc_4=\{\IC\T{-lines}\} \cup  \{\Ar\T{-lines}\} \cup  \{\EA_2\T{-lines}\} ~\#\Lc_4=3+1+4=8; \notag \db\\
&\Lc_5=\{\UG\T{-lines}\} \cup  \{\UnG_1\T{-lines}\},  ~\#\Lc_5=12+12=24;\notag \db\\ &\Lc_6=\{\EA_1\T{-lines}\} \cup  \{\EnG_1\T{-lines}\}, ~\#\Lc_6=24+24=48;\notag\db\\
&\Lc_7=\{\EA_3\T{-lines}\} \cup  \{\EnG_3\T{-lines}\}, ~\#\Lc_7=4+12=16.\notag
\end{align}

\end{corollary}

\begin{proof}
The assertions follow from Theorems \ref{th5:Gq*q=3}, \ref{th5:LineOrbitG3}.
\end{proof}

The above results for $q=3$ are illustrated by Tables \ref{tab:plane-point q=3}--\ref{tab:line-plane q=3}. The structure and notations of Tables \ref{tab:plane-point q=3}--\ref{tab:line-plane q=3} are the same described in Notation \ref{notat3:2} and at the end of Section  \ref{sec4:q=2}. In particular, the \emph{left parts} (columns $1,\ldots,6$) of Tables \ref{tab:plane-point q=3}--\ref{tab:line-plane q=3} are as follows. In the column $1$ and the top of the columns $2,\ldots,6$, see Notation \ref{notat3:2}(i), the entries correspond to Theorem \ref{th5:Gq*q=3}(ii). The order of orbits in the left parts is related to the union of the orbits under $G_3^*$ into the orbits under $G_3$.

\begin{table}[htbp]
\caption{Parameters of $\#\Ns_i^*\times\#\Ms_j^*$ submatrices $\Ic_{ij}^{\Pi\Pb*}$ and $\#\Ns_i\times\#\Ms_j$  submatrices $\Ic_{ij}^{\Pi\Pb}$ of the plane-point $40\times40$ incidence matrix $\Ic^{\Pi\Pb}$ in $\PG(3,q)$ for $q=3$ and orbits of planes $\Ns_i^*,\Ns_i$ and of points $\Ms_j^*,\Ms_j$ under $G_3^*$, $G_3$, respectively;  $4_\Gamma=(q+1)_\Gamma$}
 \centering
  \begin{tabular}
  {c|c|c|c||c|c||c||c|c|c|c}\hline
\multicolumn{6}{c||}{orbits under $G_3^*\cong PGL(2,3)$}&\multicolumn{5}{c}{orbits under $G_3\cong\mathbf{S}_4\mathbf{Z}_2^3$}\\\hline
$\Ns^*_i:\#\Ns^*_i$&$\Ms^*_4$  &$\Ms^*_1$&$\Ms^*_3$&$\Ms^*_2$&$\Ms^*_5$&$\Ns_i:\#\Ns_i$&$\Ms_4$&$\Ms_1$&$\Ms_3$&$\Ms_2$\\
$\pi$&RC&$\Cs$&TO&$4_\Gamma$&IC&$\pi$&RC&$\Cs$&TO&$4_\Gamma,\IC$\\
$\downarrow$       &12&4&8&4&12&$\downarrow$               &12&4&8&16\\\hline\hline

$\Ns^*_2: 12$      &3&2&4&1&3&$\Ns_2: 12$                  &3&2&4&4\\
 $2_\Cs$           &3&6&6&\tthr&\tthr&$2_\Cs$              &3&6&6&\tthr\\\hline

$\Ns^*_3: 4 $      &6&3&0&1&3&$\Ns_3: 4$                   &6&3&0&4\\
$3_\Cs$            &2&3&0&\tone&\tone&$3_\Cs$              &2&3&0&\tone\\\hline

$\Ns^*_5: 8$       &6&0&3&1&3&$\Ns_4: 8$                   &6&0&3&4\\
  $0_\Cs$          &4&0&3&\ttwo&\ttwo&$0_\Cs$              &4&0&3&\ttwo\\\hline\hline

$\Ns^*_1: 4$       &\bethr&\beo&\bet&\wf $\vphantom{H^{H^{H^H}}}$&\wthr&$\Ns_1: 16$       &\bethr&\beo&\bet&\wsv  \\
$\Gamma$           &1&1&1&\wf&\wo&$\Gamma,\overline{1_\Cs}$ &4&4&4&\wsv\\ \hline

$\Ns^*_4: 12$      &\bethr&\beo&\bet&\wo $\vphantom{H^{H^{H^H}}}$ &\ws &$4+12$&\multicolumn{3}{r}{\wf\pl\wthr\el\wo\pl\ws\el\wsv}&$\uparrow$\\
$\overline{1_\Cs}$ &3&3&3&\wthr&\ws&=16& \multicolumn{3}{r}{\wf\pl\wthr\el\wo\pl\ws\el\wsv}&\\\hline
 \end{tabular}
\label{tab:plane-point q=3}
\end{table}


\begin{table}[htbp]
\caption{Parameters of $\#\Lc_i^*\times\#\Ms_j^*$ submatrices $\Ic_{ij}^{\Lambda\Pb*}$ and $\#\Lc_i\times\#\Ms_j$ submatrices $\Ic_{ij}^{\Lambda\Pb}$ of the $130\times40$ line-point incidence matrix $\Ic^{\Lambda\Pb}$ in $\PG(3,q)$ for $q=3$ and orbits of lines $\Lc_i^*,\Lc_i$ and of points $\Ms_j^*,\Ms_j$ under $G_3^*$, $G_3$, respectively; $4_\Gamma=(q+1)_\Gamma$}
 \centering
  \begin{tabular} {c||c|c|c||c|c||c||c|c|c|c}\hline
\multicolumn{6}{c||}{orbits under $G_3^*\cong PGL(2,3)$}&\multicolumn{5}{c}{orbits under $G_3\cong\mathbf{S}_4\mathbf{Z}_2^3$}\\\hline
$\Lc^*_i:\#\Lc^*_i$&$\Ms^*_4$&$\Ms^*_1$&$\Ms^*_3$&$\Ms^*_2$&$\Ms^*_5$&$\Lc_i:\#\Lc_i$&$\Ms_4$&$\Ms_1$&$\Ms_3$&$\Ms_2$\\
$\lambda$ &RC&$\Cs$&TO&$4_\Gamma$&IC&$\lambda$&RC&$\Cs$&TO&$4_\Gamma,\IC$\\
$\downarrow$   &12&4&8&4&12&$\downarrow$   &12&4&8&16 \\\hline

$\Lc^*_1: 6$   &2&2&0&0&0   &$\Lc_1:6$        &2&2&0&0 \\
$\RC$          &1&3&0&\tz&\tz   &$\RC$        &1&3&0&\tz \\\hline

$\Lc^*_{12}:12$&2&0&2&0&0&$\Lc_2:12$          &2&0&2&0\\
$\EnG_2$       &2&0&3&\tz&\tz   &$\EnG_2$     &2&0&3&\tz\\\hline\hline

$\Lc^*_2: 4$   &\bez&\beo&\bet&\wo $\vphantom{H^{H^{H^H}}}$&\wz&$\Lc_3:16$  &\bez&\beo&\bet&\wo\\
$\Tr$          &0&1&1&\wo&\wz     &$\Tr,\UnG_2$&0&4&4&\wo\\\hline

$\Lc^*_6:12$   &\bez&\beo&\bet&\wz $\vphantom{H^{H^{H^H}}}$&\wo&$4+12$&\multicolumn{3}{r}{\wo\pl\wz\el\wz\pl\wo\el\wo}&$\uparrow$\\
$\UnG_2$       &0&3&3&\wz&\wo   &$=16$&\multicolumn{3}{r}{\wo\pl\wz\el\wz\pl\wo\el\wo}&\\\hline\hline

$\Lc^*_3: 3$   &\bez&\bez&\bez&\wz $\vphantom{H^{H^{H^H}}}$&\wf   &$\Lc_4:8$            &\bez&\bez&\bez&\wf \\
$\IC$          &0&0&0&\wz&\wo   &$\IC,\Ar,\EA_2$&0&0&0&\wt\\\hline

$\Lc^*_7: 1$   &\bez&\bez&\bez&\wf $\vphantom{H^{H^{H^H}}}$ &\wz      &$3+1$&\multicolumn{3}{r}{}&$\uparrow$\\
$\Ar$          &0&0&0&\wo&\wz   &$+4=8$&\multicolumn{4}{r}{\wz\pl\wf\el \wf\pl\wz\el\wo\pl\wthr\el\wf}\\\cline{1-6}

$\Lc^*_9:4$   &\bez&\bez&\bez&\wo $\vphantom{H^{H^{H^H}}}$ &\wthr      &&\multicolumn{4}{r}{\wz\pl\wo\pl\wo\el\wo\pl\wz\pl\wo\el\wt}\\
$\EA_2$       &0&0&0&\wo&\wo   &\\\hline\hline

$\Lc^*_4: 12$  &\beo&\beo&\bez&\wo $\vphantom{H^{H^{H^H}}}$ &\wo   &$\Lc_5:24$        &\beo&\beo&\bez&\wt\\
$\UG$          &1&3&0&\wthr&\wo   &$\UG,\UnG_1$      &2&6&0&\wthr\\\hline

$\Lc^*_5: 12$  &\beo&\beo&\bez&\wz $\vphantom{H^{H^{H^H}}}$&\wt   &$12+12$&\multicolumn{3}{r}{\wo\pl\wo\el\wz\pl\wt\el\wt}&$\uparrow$\\
$\UnG_1$       &1&3&0&\wz&\wt   &$=24$
&\multicolumn{3}{r}{\wthr\pl\wz\el\wo\pl\wt\el\wthr}&\\\hline\hline

$\Lc^*_8: 24$ &\beo&\bez&\beo&\wo $\vphantom{H^{H^{H^H}}}$ &\wo  &$\Lc_6:48$       &\beo&\bez&\beo&\wt\\
$\EA_1$       &2&0&3&\ws&\wt   &$\EA_1,\EnG_1$    &4&0&6&\ws\\\hline

$\Lc^*_{11}:24$&\beo&\bez&\beo&\wz $\vphantom{H^{H^{H^H}}}$&\wt&$24+24$&\multicolumn{3}{r}{\wo\pl\wo\el\wz\pl\wt\el\wt}&$\uparrow$\\
$\EnG_1$       &2&0&3&\wz&\wf   &$=48$&\multicolumn{3}{r}{\ws\pl\wz\el\wt\pl\wf\el\ws}&\\\hline\hline

$\Lc^*_{10}: 4$&\bethr&\bez&\bez&\wo $\vphantom{H^{H^{H^H}}}$ &\wz   &$\Lc_7:16$    &\bethr&\bez&\bez&\wo\\
$\EA_3$        &1&0&0&\wo&\wz   &$\EA_3,\EnG_3$    &4&0&0&\wo\\\hline

$\Lc^*_{13}:12$&\bethr&\bez&\bez&\wz $\vphantom{H^{H^{H^H}}}$&\wo   &$4+12$&\multicolumn{3}{r}{\wo\pl\wz\el\wz\pl\wo\el\wo}&$\uparrow$\\
$\EnG_3$       &3&0&0&\wz&\wo   &$=16$&\multicolumn{3}{r}{\wo\pl\wz\el\wz\pl\wo\el\wo}&\\\hline
  \end{tabular}
\label{tab:line-point q=3}
\end{table}


\begin{table}[htbp]
\caption{
Parameters of $\#\Lc_i^*\times\#\Ns_j^*$ submatrices $\Ic_{ij}^{\Lambda\Pi*}$ and $\#\Lc_i\times\#\Ns_j$ submatrices $\Ic_{ij}^{\Lambda\Pi}$ of the $130\times40$ line-point incidence matrix $\Ic^{\Lambda\Pi}$ in $\PG(3,q)$ for $q=3$ and orbits of lines $\Lc_i^*,\Lc_i$ and of planes $\Ns_j^*,\Ns_j$ under $G_3^*$, $G_3$, respectively}
 \centering
  \begin{tabular} {c||c|c|c||c|c||c||c|c||c||c}\hline
\multicolumn{6}{c||}{orbits under $G_3^*\cong PGL(2,3)$}&\multicolumn{5}{c}{orbits under $G_3\cong\mathbf{S}_4\mathbf{Z}_2^3$}\\\hline
$\Lc^*_i:\#\Lc^*_i$&$\Ns^*_2$&$\Ns^*_3$&$\Ns^*_5$&$\Ns^*_1$&$\Ns^*_4$&$\Lc_i:\#\Lc_i$&$\Ns_2$&$\Ns_3$&$\Ns_4$&$\Ns_1$\\
$\lambda$ &$2_\Cs$&$3_\Cs$&$0_\Cs$&$\Gamma$&$\overline{1_\Cs}$&$\lambda$&$2_\Cs$&$3_\Cs$&$0_\Cs$&$\Gamma,\overline{1_\Cs}$\\
$\downarrow$   &12&4&8&4&12&$\downarrow$&12&4&8&16\\\hline

$\Lc^*_1: 6$      &2&2&0&0&0     &$\Lc_1:6$&2&2&0&0\\
$\RC$             &1&3&0&\tz&\tz &$\RC$    &1&3&0&\tz\\\hline

$\Lc^*_{12}:12$   &2&0&2&0&0     &$\Lc_2:12$&2&0&2&0\\
$\EnG_2$          &2&0&3&\tz&\tz &$\EnG_2$  &2&0&3&\tz\\\hline\hline

$\Lc^*_2: 4$      &\bethr&\bez&\bez&\wo $\vphantom{H^{H^{H^H}}}$ &\wz &$\Lc_3:16$   &\bethr&\bez&\bez&\wo\\
$\Tr$             &1&0&0&\wo&\wz &$\Tr,\UnG_2$&4&0&0&\wo\\\hline

$\Lc^*_6:12$      &\bethr&\bez&\bez&\wz $\vphantom{H^{H^{H^H}}}$ &\wo&$4+12$&\multicolumn{3}{r}{\wo\pl\wz\el\wz\pl\wo\el\wo}&$\uparrow$\\
$\UnG_2$          &3&0&0&\wz&\wo&$=16$&
\multicolumn{3}{r}{\wo\pl\wz\el\wz\pl\wo\el\wo}&\\\hline\hline

$\Lc^*_3: 3$      &\bez&\bez&\bez&\wz $\vphantom{H^{H^{H^H}}}$ &\wf&$\Lc_4:8$            &\bez&\bez&\bez&\wf\\
$\IC$             &0&0&0&\wz&\wo &$\IC,\Ar,\EA_2$&0&0&0&\wt\\\hline

$\Lc^*_7: 1$      &\bez&\bez&\bez&\wf $\vphantom{H^{H^{H^H}}}$ &\wz&$3+1$&\multicolumn{3}{r}{}&$\uparrow$\\
$\Ar$             &0&0&0&\wo&\wz&$+4=8$&\multicolumn{4}{r}{\wz\pl\wf\el\wf\pl\wz\el \wo\pl\wthr\el\wf }\\\cline{1-6}

$\Lc^*_9:4$      &\bez&\bez&\bez&\wo $\vphantom{H^{H^{H^H}}}$ &\wthr&&
\multicolumn{4}{r}{\wz\pl\wo\pl\wo\el\wo\pl\wz\pl\wo\el\wt}\\
$\EA_2$          &0&0&0&\wo&\wo&\\\hline\hline

$\Lc^*_4: 12$     &\beo&\beo&\bez&\wo $\vphantom{H^{H^{H^H}}}$ &\wo &$\Lc_5:24$           &\beo&\beo&\bez&\wt\\
$\UG$             &1&3&0&\wthr&\wo &$\UG,\UnG_1$      &2&6&0&\wthr\\\hline

$\Lc^*_5: 12$     &\beo&\beo&\bez&\wz $\vphantom{H^{H^{H^H}}}$ &\wt&$12+12$&\multicolumn{3}{r}{\wo\pl\wo\el\wz\pl\wt\el\wt}&$\uparrow$\\
$\UnG_1$          &1&3&0&\wz&\wt&$=24$&\multicolumn{3}{r}{\wthr\pl\wz\el\wo\pl\wt\el\wthr}&\\\hline\hline

$\Lc^*_8:24$      &\beo&\bez&\beo&\wo $\vphantom{H^{H^{H^H}}}$ &\wo &$\Lc_6:48$           &\beo&\bez&\beo&\wt\\
 $\EA_1$          &2&0&3&\ws&\wt &$\EA_1,\EnG_1$    &4&0&6&\ws\\\hline

$\Lc^*_{11}:24$   &\beo&\bez&\beo&\wz $\vphantom{H^{H^{H^H}}}$ &\wt&$24+24$&\multicolumn{3}{r}{\wo\pl\wo\el\wz\pl\wt\el\wt}&$\uparrow$\\
$\EnG_1$          &2&0&3&\wz&\wf&$=48$&\multicolumn{3}{r}{\ws\pl\wz\el\wt\pl\wf\el\ws}&\\\hline\hline

$\Lc^*_{10}: 4$      &\bez&\beo&\bet&\wo $\vphantom{H^{H^{H^H}}}$ &\wz &$\Lc_7:16$           &\bez&\beo&\bet&\wo\\
$\EA_3$           &0&1&1&\wo&\wz &$\EA_3,\EnG_3$    &0&4&4&\wo\\\hline

$\Lc^*_{13}:12$   &\bez&\beo&\bet&\wz&\wo $\vphantom{H^{H^{H^H}}}$ &$4+12$&\multicolumn{3}{r}{\wo\pl\wz\el\wz\pl\wo\el\wo}&$\uparrow$\\
$\EnG_3$          &0&3&3&\wz&\wo&$=16$&\multicolumn{3}{r}{\wo\pl\wz\el\wz\pl\wo\el\wo}&\\\hline
\end{tabular}
\label{tab:line-plane q=3}
\end{table}

\section{Orbits and incidence matrices for $q=4$}\label{sec6:q=4}

Theorem \ref{th2:Gq q=234} states that $G_4\cong\mathbf{S}_5\cong  P\Gamma L(2,4)$. In the case $q = 4$, the five points of the twisted cubic $\Cs$ form a projective frame. Let $\xi$ be a primitive element of the field $\F_4$. The  matrix
  \begin{equation}\label{mat mprime q4}
 \Mb'''=\left[
 \begin{array}{cccc}
 1&0&0&0\\
 0&0&1&0\\
 0&1&0&0\\
 0&0&0&1
 \end{array}
  \right]
  \end{equation} swaps the points $P(\xi)$ and  $P(\xi^2)$ and fixes the other three points of $\Cs$. The matrix $\Mb'''$ is not of the form \ref{eq2:q ge5 matrixGq}.
 This explains the structure of $G_4\cong\mathbf{S}_5\cong  P\Gamma L(2,4)$: the points of $\Cs$ can be permuted by a matrix of the form \eqref{eq2:q ge5 matrixGq} or by $\Mb'''$.

\begin{theorem}\label{th6:Gq*q=4}
 Let $q=4$, $q\equiv 1\pmod 3$.

\textbf{\emph{(i)}} The group $G_4\cong\mathbf{S}_5\cong P\Gamma L(2,4)$ contains one subgroup isomorphic to $PGL(2,4)$. The matrices corresponding to the projectivities of this subgroup assume the form described by \eqref{eq2:q ge5 matrixGq}.  So, this subgroup (and only it) can be treated as a critical subgroup $G_4^*$.

\textbf{\emph{(ii)}} The plane, point, and line orbits $\Ns^*_i,\Ms^*_i$, and $\Lc^*_i$  under the critical subgroup $G_4^*$ are the same as the orbits $\Ns_i,\Ms_i$, and $\Lc_i$ in Theorems \emph{\ref{th2:Hirs}(ii)(iii)}, \emph{\ref{th2:line orbits q>=5}(i)}. We have,
\begin{align*}
&\Ns^*_1=\{\Gamma\T{-planes}\},~\#\Ns^*_1=5;~\Ns^*_{2}=\{2_\Cs\T{-planes}\},~\#\Ns^*_2=20;~\Ns^*_{3}=\{3_\Cs\T{-planes}\},\db\\
&\#\Ns^*_3=10;~ \Ns^*_{4}=\{\overline{1_\Cs}\T{-planes}\},~ \#\Ns^*_4=30;~\Ns^*_{5}=\{0_\Cs\T{-planes}\},~ \#\Ns^*_5=20.\db\\
&\Ms^*_1=\{\Cs\T{-points}\},~ \Ms^*_2=\{\Tr\T{-points}\},~ \Ms^*_3=\{3_\Gamma\T{-points}\},~ \Ms^*_4=\{1_\Gamma\T{-points}\},~\db\\
&\Ms^*_5=\{0_\Gamma\T{-points}\};~ \#\Ms^*_j=\#\Ns^*_j,~ j=1,~\ldots,~5.\\
&\Lc^*_1=\{\RC\T{-lines}\},~\Lc^*_2=\{\RA\T{-lines}\},~\#\Lc^*_1=\#\Lc^*_2=10;~\Lc^*_3=\{\Tr\T{-lines}\},~\#\Lc^*_3=5;\db \\
&\Lc^*_4=\{\IC\T{-lines}\},~ \Lc^*_5=\{\IA\T{-lines}\},~ \#\Lc^*_4=\#\Lc^*_5=6;\db\\
&\Lc^*_6=\{\UG_1\T{-lines}\},~ \#\Lc^*_6=5; ~ \Lc^*_7=\{\UG_2\T{-lines}\},~ \#\Lc^*_7=15; ~\Lc^*_8=\{\UnG\T{-lines}\}, \db\\
& \Lc^*_9=\{\EG\T{-lines}\}, ~\#\Lc^*_8= \#\Lc^*_9=60; ~\Lc^*_{10}=\{\EnG_1\T{-lines}\}, ~\Lc^*_{11}=\{\EnG_2\T{-lines}\},\db\\
&\Lc^*_{12}=\{\EnG_3\T{-lines}\},~ \#\Lc^*_{10}=\#\Lc^*_{11}=\#\Lc^*_{12}=20; ~\Lc^*_{11}=\{\EnG_2\T{-lines}\},\db\\
&\Lc^*_{12}=\{\EnG_3\T{-lines}\},~\Lc^*_{13}=\{\EnG_4\T{-lines}\},~\Lc^*_{14}=\{\EnG_5\T{-lines}\}, ~\Lc^*_{15}=\{\EnG_6\T{-lines}\},~\db\\
&\Lc^*_{16}=\{\EnG_7\T{-lines}\}, ~ \#\Lc^*_{10}=\#\Lc^*_{11}=\#\Lc^*_{12}=30.\end{align*}
\end{theorem}
 \begin{proof}
By direct computation.
\end{proof}

\begin{theorem}\label{th4:PlanePointOrbitG4}
Let $q=4$. The groups  $G_4\cong\mathbf{S}_5$ and $G^*_4$ have the same plane and point orbits.
\end{theorem}

\begin{proof}
By direct computation, no point is mapped by the matrix $\Mb'''$ into a point of another orbit.
By Theorem \ref{th2:plane=point}, $G_4$ has five orbits of planes, too.
\end{proof}

\begin{theorem}\label{th5:LineOrbitG4}
Let $q=4$. Under $G_4\cong\mathbf{S}_5$, there are $12$ line orbits $\Lc_i$ such that
\begin{align*}
&\Lc_1=\Lc^*_1,
~\Lc_2=\Lc^*_{2},
~\Lc_3=\Lc^*_7,
~\Lc_4=\Lc^*_ 8,
~\Lc_5=\Lc^*_9 ,
~\Lc_6=\Lc^*_{10},
~\Lc_7=\Lc^*_{13},\db\\
&\Lc_8=\Lc^*_{14},
~\Lc_9=\Lc^*_{3} \cup \Lc^*_{6},
~\Lc_{10}=\Lc^*_{4} \cup \Lc^*_{5},
~\Lc_{11}=\Lc^*_{11} \cup \Lc^*_{12},
~\Lc_{12}=\Lc^*_{15} \cup \Lc^*_{16},
\end{align*}

\end{theorem}

\begin{proof}
The matrix $ \Mb'''$  permutes the points of $\Cs$, so it maps a real chord into itself. The matrix   $ \Mb'''$  maps the $\Tr$-line $x_2 = 0$, $x_3 = 0$ into the $\UG_1$-line $x_1 = 0$, $x_3 = 0$,
maps the $\IC$-line $x_0 + x_1 = \xi x_3$, $x_0 + \xi^2 x_1  =  x_2  $ into the $\IA$-line $x_0 + x_2 = \xi x_3$,  $x_0 +  x_1=  \xi^2 x_2 $,
maps the $\EnG_2$-line $x_2 = 0$, $ x_0 = \xi x_3$ into the $\EnG_3$-line $x_1 = 0$, $ x_0 = \xi x_3$,
maps the $\EnG_6$-line $x_0 = x_2$, $x_0 + \xi^2 x_3$, into the $\EnG_7$-line $x_0 = x_1$, $x_0 + \xi^2 x_3$.
By direct computation, no other line orbits merge.
 \end{proof}

\begin{corollary}
Let $q=4$. The line orbits under $G_4\cong\mathbf{S}_5$ are as follows.
\begin{align}
& \Lc_1=\{\RC\T{-lines}\}, ~\#\Lc_1=10;
~\Lc_2= \{\RA\T{-lines}\}, ~\#\Lc_2=10;   ~\Lc_3=\{\UG_2\T{-lines}\},\notag \db\\
&\#\Lc_3=15; ~\Lc_4=\{\UnG\T{-lines}\}, ~\#\Lc_4=60;  ~\Lc_5=\{\EG\T{-lines}\}, ~\#\Lc_5=60;\notag \db\\
&\Lc_6=\{\EnG_1\T{-lines}\}, ~\#\Lc_6=20; ~\Lc_7=\{\EnG_4\T{-lines}\}, ~\#\Lc_7=30; ~\Lc_8= \{\EnG_5\T{-lines}\}, \notag \db\\
&\#\Lc_8=30;  ~\Lc_9=\{\Tr\T{-lines}\} \cup  \{\UG_1\T{-lines}\}, ~\#\Lc_{9}=10;
~\Lc_{10}=\{\IC\T{-lines}\} \cup  \{\IA\T{-lines}\},  \notag \db\\
&\#\Lc_{10}=12; ~\Lc_{11}=\{\EnG_2\T{-lines}\} \cup  \{\EnG_3\T{-lines}\}, ~\#\Lc_{11}=40;   \notag \db\\
&~\Lc_{12}= \{\EnG_6\T{-lines}\} \cup  \{\EnG_7\T{-lines}\}, ~\#\Lc_{12}=60.
\end{align}

\end{corollary}

\begin{proof}
The assertions follow from Theorems \ref{th6:Gq*q=4}, \ref{th5:LineOrbitG4}.
\end{proof}

The above results for $q=4$ are illustrated by Tables \ref{tab:plane-point q=4}--\ref{tab:line-plane q=4}. The structure and notations of Tables \ref{tab:plane-point q=4}--\ref{tab:line-plane q=4} are the same described in Notation \ref{notat3:2} and at the end of Section~\ref{sec4:q=2}. In particular,
the \emph{left parts} (columns $1,\ldots,6$) of Tables \ref{tab:plane-point q=3}--\ref{tab:line-plane q=3} are as follows.
In the column $1$ and the top of the columns $2,\ldots,6$, see Notation \ref{notat3:2}(i), the entries correspond to Theorem \ref{th6:Gq*q=4}(ii). The order of orbits in the left parts is related to the union of the orbits under $G_4^*$ into the orbits under $G_4$.

\begin{table}[htbp]
\caption{Parameters of $\#\Ns_i^*\times\#\Ms_j^*$ submatrices $\Ic_{ij}^{\Pi\Pb*}$ and $\#\Ns_i\times\#\Ms_j$  submatrices $\Ic_{ij}^{\Pi\Pb}$ of the plane-point $85\times85$ incidence matrix $\Ic^{\Pi\Pb}$ in $\PG(3,q)$ for $q=4$ and orbits of planes $\Ns_i^*,\Ns_i$ and of points $\Ms_j^*,\Ms_j$ under $G_4^*\cong PGL(2,4)$, $G_4\cong\mathbf{S}_5\cong P\Gamma L(2,4)$}
 \centering
  \begin{tabular}
  {c|c|c|c|c|c}\hline
$\Ns^*_i=\Ns_i:\#\Ns_i^*$  &$\Ms^*_1=\Ms_1$&$\Ms^*_2=\Ms_2$&$\Ms^*_3=\Ms_3$&$\Ms^*_4=\Ms_4$&$\Ms^*_5=\Ms_5$ \\
$\pi$&$\Cs$&T&$3_\Gamma$&$1_\Gamma$&$0_\Gamma$  \\
$\downarrow$         &5&20&10&30&20      \\\hline\hline

$\Ns^*_1=\Ns_1$: 5&1&8&6&6&0 \\
$\Gamma$&1&2&3&1&0\\ \hline

$\Ns_2^*=\Ns_2$: 20&2&7&1&6&5\\
$2_\Cs$ &8&7&2&4&5\\\hline

$\Ns_3^*=\Ns_3$: 10&3&2&4&6&6\\
$3_\Cs$&6&1&4&2&3\\\hline

$\Ns_4^*=\Ns_4$: 30&1&4&2&10&4\\
$\overline{1_\Cs}$&6&6&6&10&6\\\hline

$\Ns_5^*=\Ns_5$: 20&0&5&3&6&7\\
$0_\Cs$&0&5&6&4&7\\\hline
 \end{tabular}
\label{tab:plane-point q=4}
\end{table}

\begin{table}[htbp]
\caption{Parameters of $\#\Lc^*_i\times\#\Ms_j^*$ submatrices $\Ic_{ij}^{\Lambda\Pb*}$  and $\#\Lc_i\times\#\Ms_j$
submatrices $\Ic_{ij}^{\Lambda\Pb}$ of the line-point $357\times85$ incidence matrix $\Ic^{\Lambda\Pb}$ in $\PG(3,q)$ for $q=4$ and orbits of lines $\Lc^*_i$, $\Lc_i$ and points $\Ms_j^*$, $\Ms_j$, under $G^*_4$, $G_4$, respectively; $\Ms_j^*=\Ms_j$}
 \centering
 \renewcommand\arraystretch{0.93}
  \begin{tabular}
  {c||c|c|c|c|c||c||c|c|c|c|c}\hline
\multicolumn{6}{c||}{orbits under $G_4^*\cong PGL(2,4)$}&
\multicolumn{6}{c}{orbits under $G_4\cong\mathbf{S}_5\cong P\Gamma L(2,4)$}\\\hline
$\Lc^*_i:\#\Lc^*_i$&$\Ms_1^*$&$\Ms_2^*$&$\Ms_3^*$&$\Ms_4^*$&$\Ms_5^*$&$\Lc_i:\#\Lc_i$&$\Ms_1$&$\Ms_2$&$\Ms_3$&$\Ms_4$&$\Ms_5$\\
$\lambda$&$\Cs$&T&$3_\Gamma$&$1_\Gamma$&$0_\Gamma$&$\lambda$&$\Cs$&T&$3_\Gamma$&$1_\Gamma$&$0_\Gamma$\\
$\downarrow$        &5&20&10&30&20&$\downarrow$        &5&20&10&30&20      \\
\hline\hline

$\Lc^*_1:10$   &2&0&1&0&2  &$\Lc_1:10$   &2&0&1&0&2  \\
$\RC$          &4&0&1&0&1  &$\RC$        &4&0&1&0&1  \\\hline

$\Lc^*_2:10$   &0&2&3&0&0  &$\Lc_2:10$   &0&2&3&0&0  \\
$\RA$          &0&1&3&0&0  &$\RA$        &0&1&3&0&0  \\ \hline

$\Lc^*_7: 15$  &1&0&2&2&0  &$\Lc_3: 15$  &1&0&2&2&0   \\
$\UG_2$        &3&0&3&1&0  &$\UG_2$      &3&0&3&1&0   \\ \hline

$\Lc^*_8:60$   &1&1&0&2&1  &$\Lc_4:60$   &1&1&0&2&1   \\
$\UnG$        &12&3&0&4&3  &$\UnG$      &12&3&0&4&3  \\ \hline

$\Lc^*_9: 60$  &0&2&1&2&0  &$\Lc_5: 60$  &0&2&1&2&0  \\
$\EG$          &0&6&6&4&0  &$\EG$        &0&6&6&4&0  \\  \hline

$\Lc^*_{10}:20$&0&1&1&0&3  &$\Lc_6:20$   &0&1&1&0&3   \\
$\EnG_1$       &0&1&2&0&3  &$\EnG_1$     &0&1&2&0&3   \\\hline

$\Lc^*_{13}:30$&0&0&1&2&2  &$\Lc_7: 30$  &0&0&1&2&2  \\
$\EnG_4$       &0&0&3&2&3  &$\EnG_4$     &0&0&3&2&3  \\\hline

$\Lc^*_{14}:30$&0&0&1&2&2  &$\Lc_8: 30$  &0&0&1&2&2  \\
$\EnG_5$       &0&0&3&2&3  &$\EnG_5$     &0&0&3&2&3  \\\hline\hline

$\Lc^*_3: 5$   &$\vphantom{H^{H^{H^H}}}$\beo&\bef&\bez&\bez&\bez  &$\Lc_9:10$    &\beo&\bef&\bez&\bez&\bez  \\
$\Tr$          &1&1&0&0&0  &$\Tr,\UG_1$&2&2&0&0&0  \\\hline

$\Lc^*_6:5$    &$\vphantom{H^{H^{H^H}}}$\beo&\bef&\bez&\bez&\bez  &$5+5=10$  \\
$\UG_1$        &1&1&0&0&0  &  \\ \hline\hline

$\Lc^*_4:6$    &$\vphantom{H^{H^{H^H}}}$\bez&\bez&\bez&\befv&\bez  &$\Lc_{10}:12$ &\bez&\bez&\bez&\befv&\bez  \\
$\IC$          &0&0&0&1&0  &$\IC,\IA$  &0&0&0&2&0  \\\hline

$\Lc^*_5:6$    &$\vphantom{H^{H^{H^H}}}$\bez&\bez&\bez&\befv&\bez  &$6+6=12$  \\
$\IA$          &0&0&0&1&0  &  \\\hline\hline

$\Lc^*_{11}:20$&$\vphantom{H^{H^{H^H}}}$\bez&\beo&\bez&\bethr&\beo  &$\Lc_{11}: 40$    &\bez&\beo&\bez&\bethr&\beo  \\
$\EnG_2$       &0&1&0&2&1  &$\EnG_2,\EnG_3$&0&2&0&4&2  \\\hline

$\Lc^*_{12}:20$&$\vphantom{H^{H^{H^H}}}$\bez&\beo&\bez&\bethr&\beo  &$20+20$  \\
$\EnG_3$       &0&1&0&2&1  &$=40$  \\\hline\hline

$\Lc^*_{15}:30$&$\vphantom{H^{H^{H^H}}}$\bez&\bet&\bez&\beo&\bet  &$\Lc_{12}: 60$    &\bez&\bet&\bez&\beo&\bet  \\
$\EnG_6$       &0&3&0&1&3  &$\EnG_6,\EnG_7$&0&6&0&2&6  \\\hline

$\Lc^*_{16}:30$&$\vphantom{H^{H^{H^H}}}$\bez&\bet&\bez&\beo&\bet  &$30+30$  \\
$\EnG_7$       &0&3&0&1&3  &$=60$  \\\hline
 \end{tabular}
\label{tab:line-point q=4}
\end{table}

\begin{table}[htbp]
\caption{Parameters of $\#\Lc^*_i\times\#\Ns_j^*$ submatrices $\Ic_{ij}^{\Lambda\Pi*}$  and $\#\Lc_i\times\#\Ns_j$
submatrices $\Ic_{ij}^{\Lambda\Pi}$ of the line-plane $357\times85$ incidence matrix $\Ic^{\Lambda\Pi}$ in $\PG(3,q)$ for $q=4$ and orbits of lines $\Lc^*_i$, $\Lc_i$ and planes $\Ns_j^*$, $\Ns_j$, under $G^*_4$, $G_4$, respectively; $\Ns_j^*=\Ns_j$}
 \centering
  \renewcommand\arraystretch{0.93}
  \begin{tabular}
  {c||c|c|c|c|c||c||c|c|c|c|c}\hline
\multicolumn{6}{c||}{orbits under $G_4^*\cong PGL(2,4)$}&
\multicolumn{6}{c}{orbits under $G_4\cong\mathbf{S}_5\cong P\Gamma L(2,4)$}\\\hline
$\Lc^*_i:\#\Lc^*_i$&$\Ns_1^*$&$\Ns_2^*$&$\Ns_3^*$&$\Ns_4^*$&$\Ns_5^*$&$\Lc_i:\#\Lc_i$&$\Ns_1$&$\Ns_2$&$\Ns_3$&$\Ns_4$&$\Ns_5$\\
$\lambda$&$\Cs$&T&$3_\Gamma$&$1_\Gamma$&$0_\Gamma$&$\lambda$&$\Cs$&T&$3_\Gamma$&$1_\Gamma$&$0_\Gamma$\\
$\downarrow$        &5&20&10&30&20&$\downarrow$        &5&20&10&30&20      \\\hline\hline

$\Lc^*_1:10$     &0&2&3&0&0&$\Lc_1:10$ &0&2&3&0&0    \\
$\RC$            &0&1&3&0&0&$\RC$      &0&1&3&0&0    \\\hline

$\Lc^*_2:10$     &2&0&1&0&2&$\Lc_2:10$ &2&0&1&0&2    \\
$\RA$            &4&0&1&0&1&$\RA$      &4&0&1&0&1    \\ \hline

$\Lc^*_7: 15$    &1&0&2&2&0&$\Lc_3:15$ &1&0&2&2&0    \\
$\UG_2$          &3&0&3&1&0&$\UG_2$    &3&0&3&1&0     \\ \hline

$\Lc^*_8:60$     &0&2&1&2&0&$\Lc_4:60$ &0&2&1&2&0     \\
$\UnG$           &0&6&6&4&0&$\UnG$     &0&6&6&4&0   \\ \hline

$\Lc^*_9: 60$    &1&1&0&2&1&$\Lc_5:60$ &1&1&0&2&1  \\
$\EG$           &12&3&0&4&3&$\EG$      &12&3&0&4&3   \\  \hline

$\Lc^*_{10}:20$  &0&1&1&0&3&$\Lc_6:20$ &0&1&1&0&3     \\
$\EnG_1$         &0&1&2&0&3&$\EnG_1$   &0&1&2&0&3     \\\hline

$\Lc^*_{13}:30$  &0&0&1&2&2&$\Lc_7: 30$&0&0&1&2&2    \\
$\EnG_4$         &0&0&3&2&3&$\EnG_4$   &0&0&3&2&3    \\\hline

$\Lc^*_{14}:30$  &0&0&1&2&2&$\Lc_8:30$ &0&0&1&2&2   \\
$\EnG_5$         &0&0&3&2&3&$\EnG_5$   &0&0&3&2&3    \\\hline\hline

$\Lc^*_3: 5$     &$\vphantom{H^{H^{H^H}}}$\beo&\bef&\bez&\bez&\bez&$\Lc_9:10$    &\beo&\bef&\bez&\bez&\bez    \\
$\Tr$            &1&1&0&0&0&$\Tr,\UG_1$&2&2&0&0&0 \\\hline

$\Lc^*_6:5$      &$\vphantom{H^{H^{H^H}}}$\beo&\bef&\bez&\bez&\bez&$5+5=10$  \\
$\UG_1$          &1&1&0&0&0&  \\ \hline\hline

$\Lc^*_4:6$      &$\vphantom{H^{H^{H^H}}}$\bez&\bez&\bez&\befv&\bez&$\Lc_{10}:12$ &\bez&\bez&\bez&\befv&\bez  \\
$\IC$            &0&0&0&1&0&$\IC,\IA$  &0&0&0&2&0  \\\hline

$\Lc^*_5:6$      &$\vphantom{H^{H^{H^H}}}$\bez&\bez&\bez&\befv&\bez&$6+6=12$  \\
$\IA$            &0&0&0&1&0&  \\\hline\hline

$\Lc^*_{11}:20$  &$\vphantom{H^{H^{H^H}}}$\bez&\beo&\bez&\bethr&\beo&$\Lc_{11}:40$      &\bez&\beo&\bez&\bethr&\beo\\
$\EnG_2$         &0&1&0&2&1&$\EnG_2,\EnG_3$ &0&2&0&4&2 \\\hline

$\Lc^*_{12}:20$  &$\vphantom{H^{H^{H^H}}}$\bez&\beo&\bez&\bethr&\beo&$20+20$  \\
$\EnG_3$         &0&1&0&2&1&$=40$  \\\hline

$\Lc^*_{15}:30$  &$\vphantom{H^{H^{H^H}}}$\bez&\bet&\bez&\beo&\bet&$\Lc_{12}: 60$     &\bez&\bet&\bez&\beo&\bet \\
$\EnG_6$         &0&3&0&1&3&$\EnG_6,\EnG_7$ &0&6&0&2&6 \\\hline

$\Lc^*_{16}:30$  &$\vphantom{H^{H^{H^H}}}$\bez&\bet&\bez&\beo&\bet&$30+30$  \\
$\EnG_7$         &0&3&0&1&3&$=60$  \\\hline
 \end{tabular}
\label{tab:line-plane q=4}
\end{table}

\section*{Acknowledgments}The research of A.A. Davydov was carried out within the state assignment of Ministry of Science and Higher Education of the Russian Federation  No.\ FFNU-2025-0028 for Kharkevich Institute for Information Transmission Problems.
The research of S. Marcugini and F. Pambianco was supported in part by the Italian
National Group for Algebraic and Geometric Structures and their Applications (GNSAGA -
INDAM) (Contract No. U-UFMBAZ-2019-000160, 11.02.2019) and by University of Perugia
(Project No. 98751: Strutture Geometriche, Combinatoria e loro Applicazioni, Base Research
Fund 2017--2019; Fighting Cybercrime with OSINT, Research Fund 2021). This work was partially funded by the SERICS project (PE00000014) under the MUR National Recovery and Resilience Plan funded by the European Union-NextGenerationEU.\\
A.A. Davydov would like to thank E. Jagudaeva for organizational support.

\end{document}